\renewcommand\H{\mathrm{H}}
\renewcommand\L{\mathrm{L}}
\newcommand\W{\mathrm{W}}
\newcommand\R{\mathbb{R}}

\newcommand\bbL{\mathbb{L}}
\newcommand\bH{\mathbf{H}}
\newcommand\bL{\mathbf{L}}
\newcommand\cC{\mathcal{C}}
\newcommand\cD{\mathcal{D}}
\newcommand\cA{\mathcal{A}}
\newcommand\cT{\mathcal{T}}
\newcommand\cP{\mathcal{P}}
\newcommand\cV{\mathcal{V}}
\renewcommand\t{\mathtt{t}}
\newcommand\bn{\boldsymbol{n}}
\newcommand\bF{\boldsymbol{f}}
\newcommand\bu{\boldsymbol{u}}
\newcommand\bv{\boldsymbol{v}}
\newcommand\br{\boldsymbol{r}}
\newcommand\bs{\boldsymbol{s}}
\newcommand\be{\boldsymbol{e}}
\newcommand\beps{\boldsymbol{\varepsilon}}
\newcommand\bsig{\boldsymbol{\sigma}}
\newcommand\btau{\boldsymbol{\tau}}
\newcommand\bze{\boldsymbol{\zeta}}
\newcommand\bet{\boldsymbol{\eta}}
\newcommand\bgam{\boldsymbol{\gamma}}

\newcommand\bchi{\boldsymbol{\chi}}
\newcommand\kq{\mathfrak{q}}
\newcommand\kp{\mathfrak{p}}
\newcommand\bdiv{\mathop{\mathbf{div}}\nolimits}

\newcommand{\norm}[1]{\left\| #1\right\|}
\newcommand\tr{\mathop{\mathrm{tr}}\nolimits}
\newcommand{\set}[1]{\left\{ #1\right\}}

\newcommand\disp{\displaystyle}

\documentclass[11pt]{article}
\usepackage{amsmath,amsfonts,amsthm, amssymb, latexsym,graphics, color}
\usepackage{booktabs, multirow}

\DeclareMathOperator*{\esssup}{ess\,sup}
\newtheorem{lemma}{Lemma}[section]
\newtheorem{theorem}{Theorem}[section]

\newtheorem{corollary}{Corollary}[section]

\numberwithin{equation}{section}
\numberwithin{figure}{section}
\numberwithin{table}{section}

\date{}
\title{A mixed finite element method with reduced symmetry \\
for the standard  model in linear viscoelasticity
\thanks{This research was partially supported by Spain's Ministry of Economy Project MTM2017-87162-P; by CONICYT-Chile through the project AFB170001 of the PIA Program: Concurso Apoyo a Centros Cient\'ificos y Tecnol\'ogicos de Excelencia con Financiamiento Basal; and by Centro de Investigaci\'on en Ingenier\'ia Matem\'atica (CI$^2$MA), Universidad de Concepci\'on.}}

\author{{\sc Gabriel N. Gatica}\thanks{CI$^2$MA and Departamento de Ingenier\'\i a Matem\' atica,
Universidad de Concepci\' on, Casilla 160-C, Concepci\' on, Chile, email: {\tt ggatica@ci2ma.udec.cl}}\quad {\sc Antonio M\'arquez}\thanks{Departamento de Construcci\'on e Ingenier\'\i a de Fabricaci\'on, Universidad de Oviedo, Oviedo, Espa\~na, e-mail: {\tt amarquez@uniovi.es}.}
\quad
{\sc Salim Meddahi}\thanks{Departamento de Matem\'aticas, Facultad de Ciencias, Universidad de Oviedo, Calvo Sotelo s/n, Oviedo, Espa\~na, e-mail: {\tt salim@uniovi.es}.}}

\begin{document}

\maketitle

\begin{abstract}
\noindent
We introduce and analyze a new mixed finite element method with reduced symmetry for  the standard linear model in viscoelasticity. Following a previous approach employed for linear elastodynamics, the present problem is formulated as a second-order hyperbolic partial differential equation in which, after using the motion equation to eliminate the displacement unknown, the stress tensor remains as the main variable to be found. The resulting variational formulation is shown to be well-posed,  and a class of  $\H(\text{div})$-conforming semi-discrete schemes is proved to be convergent. Then, we use the Newmark trapezoidal rule to obtain an associated fully discrete scheme, whose main convergence results are also established. Finally, numerical examples illustrating the performance of the method are reported.

\end{abstract}
\medskip

\noindent
\textbf{Mathematics Subject Classification.} 65N30, 65M12, 65M15, 74H15
\medskip

\noindent
\textbf{Keywords.}
mixed finite elements, elastodynamics, error estimates

\section{Introduction}  

We are interested in the problem of wave propagation in solids exhibiting a linear viscoelastic behaviour. Viscoelastic materials are characterized by reactions involving a combination of elastic and viscous effects when mechanical loads are applied to them. The  viscoelastic properties of solids are modeled through constitutive equations relating the strain and the stress tensors. In the linear case, it is shown in \cite{gurtin} that these constitutive relations have an integral form, which reflects that the stress depends on the history of the strain evolution, and an equivalent differential form  represented in one dimension by different arrangements  of springs and dashpots. In this paper we are concerned with the standard linear solid model, also known as the Zener model. It consists in a parallel  combination of one spring and one Maxwell component (a serial combination of one spring and one dashpot). It is the simplest model for viscoelasticity that  takes into account important phenomena such as recovery and stress relaxation, see \cite{salencon} for more details.  

The non-local version of the constitutive law can be used to eliminate the stress tensor and  formulate the linear viscoelastic system solely in terms of the displacement field. This approach gives rise to an integro-differential weak formulation  whose mathematical analysis can be found in \cite{Morro}. Most of the numerical work on linear viscoelasticity focused on this formulation. It has been used in various contexts by engineers, in spite of the negative impact that the Volterra integral term have on the computational performance.  An overview of the different numerical techniques used to solve this problem can be found in \cite{Marques}. Convergence of schemes based on continuous and discontinuous Galerkin finite elements in space and quadratures in time have also been explored in \cite{Shaw, Shaw1, riv1} for the displacement formulation. We also refer to \cite{Idesman, Janovsky, riv2} for other studies of numerical methods for linear viscoelasticity. 

In this paper, we are interested in formulations based exclusively on differential equations and relying on the stress tensor as primary unknown. To our knowledge, B\'ecache \textsl{et al.} \cite{Becache} introduced the first mixed formulation for viscoelastic wave problems employing an $\mathrm H(\mathrm{div})$-energy space for the stress. Their numerical scheme combines an explicit time quadrature  with  a space discretization based on the mixed finite element introduced in \cite{Becache0} for linear elastodynamics. The resulting numerical method delivers low order symmetric approximations of the stress on regular cubical grids. Rognes and Winther reinforced  this strategy by analyzing in  \cite{rognes} mixed formulations for the quasi-static Maxwell and Kelvin-Voigt models with a weak symmetry restriction on the stress tensor. This approach gives rise to mixed variational formulations for linear viscoelasticity whose spacial discretizations can be built upon stable families of simplicial finite elements designed for mixed approximations of the elasticity system with reduced symmetry, see for example \cite{abd, ArnoldFalkWinther, Becache0, CGG, GG, stenberg} and the references therein. Recently, this strategy has  been generalized by Lee \cite{lee} for the dynamic standard linear solid model. 

We point out that the articles \cite{Becache, rognes, lee} carry out convergence analyses for semi-discrete schemes leading to $\mathrm L^2$-error estimates for the stress tensor.  Our aim here is to introduce    semi and fully discrete versions of a new mixed formulation for the  standard linear solid model and to prove optimal convergence rates for the stress tensor in the full energy norm, namely, in the  $\mathrm H(\text{div})$-norm. Our approach is based on the mixed formulation introduced in \cite{ggm-JSC-2017} for linear elastodynamics. We show that this formulation can be adapted to deal with our viscoelastic model problem  on general domains, including heterogeneous media and general boundary conditions. We analyze the continuous problem and provide convergence analyses for the semi-discrete and fully discrete problems by using fairly standard discrete energy decay techniques. It is also  worthwhile to mention that, although we only maintain the stress tensor as primary unknown (besides the rotation), accurate approximations of the acceleration field can be  directly obtained from the linear momentum equation. 

This article is structured as follows: Section~\ref{section2} is devoted to notations,  definitions, and basic results that are used throughout the document. In Section~\ref{section3}, we introduce a new mixed variational formulation for the standard linear solid model. Next, we recall in Section~\ref{section4} the main properties of the Arnold-Falk-Winther \cite{ArnoldFalkWinther} family of mixed finite elements and use them to construct in Section~\ref{section5} a space discretization of the variational problem. Then, we  employ a Galerkin procedure to prove the existence of weak solutions. The convergence of the  semi-discrete  problem is carried out in Section~\ref{section6}. We propose in Section~\ref{section7} a fully discrete method based on an implicit Newmark scheme and undertake its   stability and convergence analysis. Finally, we confirm in Section~\ref{section8} the theoretical rates of convergence by showing results obtained from a series of numerical tests.

\section{Notations and preliminary results}\label{section2}
We dedicate this section to provide part of the notations, definitions, and preliminary results that will be employed along the paper. We first denote by $\boldsymbol{I}$ the identity matrix of $\R^{d \times d}$ ($d=2,3$), and by $\mathbf{0}$ the null vector in $\R^d$ or the null tensor in $\R^{d \times d}$. In addition, the component-wise inner product of two matrices $\bsig, \,\btau \in\R^{d \times d}$ is defined by $\bsig:\btau:= \tr( \bsig^{\t} \btau)$, where $\tr(\btau):=\sum_{i=1}^d\tau_{ii}$ and $\btau^{\t}:=(\tau_{ji})$ stand for the trace and the transpose of $\btau = (\tau_{ij})$, respectively. In turn, for $\bsig:\Omega\to \R^{d \times d}$ and $\bu:\Omega\to \R^d$, we set the row-wise divergence $\bdiv \bsig:\Omega \to \R^d$, the  row-wise gradient $\nabla \bu:\Omega \to \R^{d \times d}$, and the strain tensor $\beps(\bu) : \Omega \to \R^{d \times d}$ as
\[
(\bdiv \bsig)_i := \sum_j   \partial_j \sigma_{ij} \,, \quad (\nabla \bu)_{ij} := \partial_j u_i, \quad\text{and}\quad 
\beps(\bu) := \frac{1}{2}\Big\{\nabla\bu+(\nabla\bu)^{\t}\Big\}\,,
\]
respectively. Next, we let $\Omega$ be a polyhedral Lipschitz bounded domain of $\R^d$ $(d=2,3)$, with boundary $\partial\Omega$. Furthermore, for $s\in \mathbb{R}$, $\norm{\cdot}_{s,\Omega}$ stands indistinctly for the norm of the Hilbertian Sobolev spaces $\H^s(\Omega)$, $\bH^s(\Omega):= \H^s(\Omega)^d$ or $\mathbb{H}^s(\Omega):= \H^s(\Omega)^{d \times d}$, with the convention $\H^0(\Omega):=\L^2(\Omega)$. In all what follows, $(\cdot, \cdot)$ stands for the inner product in $\L^2(\Omega)$, $\bL^2(\Omega):=\L^2(\Omega)^d$, $\mathbb{L}^2(\Omega):=\L^2(\Omega)^{d \times d}$, and $\mathfrak{L}^2(\Omega) := \mathbb{L}^2(\Omega) \times \mathbb{L}^2(\Omega)$. We also introduce the Hilbert space  $\mathbb{H}(\bdiv, \Omega):=\big\{\btau\in\mathbb{L}^2(\Omega):\ \bdiv\btau\in\bL^2(\Omega)\big\}$ and denote the corresponding norm $\norm{\btau}^2_{\mathbb H(\bdiv,\Omega)}:=\norm{\btau}_{0,\Omega}^2+\norm{\bdiv\btau}^2_{0,\Omega}$. 
 
Since we will deal with a time-domain problem, besides the Sobolev spaces defined above, we need to introduce spaces of functions acting on a bounded time interval $(0,T)$ and with values in a separable Hilbert space $V$, whose norm is denoted here by $\norm{\cdot}_{V}$. In particular, for $1 \leq p\leq \infty$, $\L^p(V)$ is the space of classes of functions $f:\ (0,T)\to V$ that are B\"ochner-measurable and such that $\norm{f}_{\L^p(V)}<\infty$, with 
\[
\norm{f}^p_{\L^p(V)}:= \int_0^T\norm{f(t)}_{V}^p\, \text{d}t\quad \hbox{for $1\leq p < \infty$},
\quad\hbox{and} \quad \norm{f}_{\L^\infty(V)}:= \esssup_{[0, T]} \norm{f(t)}_V.
\]
We use the notation $\mathcal{C}^0(V)$ for the Banach space consisting of all continuous functions $f:\ [0,T]\to V$. More generally, for any $k\in \mathbb{N}$, $\mathcal{C}^k(V)$ denotes the subspace of $\mathcal{C}^0(V)$ of all functions $f$ with (strong) derivatives $\frac{\text{d}^j f}{\text{d}t^j}$ in $\mathcal{C}^0(V)$ for all $1\leq j\leq k$. In what follows, we will use indistinctly the notations $\dot{f}:= \frac{\text{d} f}{\text{d}t}$ and $\ddot{f} := \frac{\text{d}^2 f}{\text{d}t^2} $ to express the first and second derivatives with respect to the variable $t$. Furthermore,  we will use the Sobolev space
\[
\begin{array}{c}
\W^{1, p}(V):= \left\{f: \ \exists g\in \L^p( V)
\ \text{and}\ \exists f_0\in V\ \text{such that}\
 f(t) = f_0 + \int_0^t g(s)\, \text{d}s\quad \forall t\in [0,T]\right\},
\end{array}
\]
and denote $\H^1(V):= \W^{1,2}(V)$. The space $\W^{k, p}(V)$ is defined recursively  for all $k\in\mathbb{N}$.

On the other hand, given two Hilbert spaces $S$ and $Q$ and a bounded bilinear form $a:S\times Q\to\R$, we denote $\ker(a):=\set{s\in S:\ a(s,q)=0\ \forall\, q\in Q}$.
We say that $a$ satisfies the inf-sup condition for the pair $\{ S,Q\}$, whenever there exists $\kappa>0$ such that
\begin{equation}\label{inf-sup-S-Q}
\sup_{0\neq s\in S}\frac{a(s,q)}{\norm{s}_{ S}}
\ge \kappa \, \norm{q}_{Q}
\qquad\forall\, q\in Q\,.
\end{equation}
We will repeatedly use the well-known fact  that (see \cite{BoffiBrezziFortinBook}) if $a$ satisfies the inf-sup condition for the pair $\{ S, Q\}$ and if $\ell\in S'$ vanishes identically on $\ker(a)$, then there exists a unique $q\in Q$ such that 
\[
a(s,q) = \ell(s) \quad \forall\, s\in S\,.
\]

Throughout the rest of the paper, given any positive expressions $X$ and $Y$ depending on the meshsize  $h$ of a triangulation, the notation $X \,\lesssim\, Y$  means that $X \,\le\, C\, Y$ with a constant $C > 0$ independent of the mesh size $h$ and the time discretization step $\Delta t$.

\section{A mixed formulation of the Zener model}\label{section3}
 
In what follows, $\Omega\subset \mathbb R^d$ ($d=2,3$) is a  polyhedral Lipschitz domain representing a viscoelastic body with a piecewise constant mass density $\rho$. We assume that there exists a polygonal/polyhedral disjoint partition $\big\{\Omega_j,\ j= 1,\ldots,J\big\}$ of  $\bar \Omega$  such that $\rho|_{\Omega_j}:= \rho_j>0$ for  $j=1,\ldots,J$. In addition, we assume that the boundary $\partial \Omega$ admits a disjoint partition $\partial \Omega = \Gamma_D\cup \Gamma_N$, and denote its outward unit normal vector by $\bn$. Then, given a body force $\bF:\Omega\times[0, T] \to \R^d$, we seek a displacement field $\bu:\Omega\times[0, T] \to \R^d$ and a stress tensor $\bsig:\Omega\times[0,T]\to \R^{d\times d}$ satisfying the equations of linear viscoelasticity with a strain-to-stress relationship given by Zener's material law \cite{salencon}:
\begin{align}\label{model}
\begin{split}
\rho\ddot{\bu}-\bdiv\bsig &=\bF \quad \text{in $\Omega\times (0, T]$},
\\[1ex]
 \bsig +  \omega \dot{\bsig} &= \cC \beps(\bu) + \omega \cD \beps(\dot{\bu}) \quad \text{in $\Omega\times (0, T]$},
 \\[1ex]
 \bu &= \mathbf{0}  \quad \text{on $\Gamma_D\times (0, T]$},
 \\[1ex]
 \bsig\bn &= \mathbf{0}  \quad \text{on $\Gamma_N\times (0, T]$},
\end{split} 
\end{align}
where $\cC$ and $\cD$ are two  space-dependent symmetric and positive definite tensors of order 4 and $\omega\in L^\infty(\Omega)$ is such that $\omega(x) \geq \omega_0 >0$ \textit{a.e.} in $\Omega$. In order to obtain a dissipative model, we assume that the tensor $\cD - \cC$, which corresponds to the diffusive part of the elastic model, is also positive definite, cf. \cite{Becache}. In turn, the model problem \eqref{model} is assumed to be subject to the initial conditions
\begin{equation}\label{init1}
	\bu(0) = \bu_0, \quad 
 \dot{\bu}(0) = \bu_1,    \quad \text{and} \quad 
 \bsig(0) = \bsig_0 \qquad\text{in $\Omega$}.
\end{equation}

The starting point of the procedure leading to a mixed formulation for \eqref{model} is  a decomposition of the stress tensor $\bsig$ into a purely elastic component $\bgam:= \cC \beps(\bu)$ and a Maxwell  component $\bze:= \bsig - \bgam$, cf. \cite{Becache, rognes, lee}. With these notations, the second equation of \eqref{model} can be rewritten as
\[
\dot\bsig \,=\, \cD \beps(\dot{\bu}) \,-\, \omega ^{-1} \bze \quad \text{in  $\Omega\times (0, T]$},
\]
which yields
\[
\ddot\bze \,=\, \ddot\bsig - \cC \beps(\ddot{\bu}) \,=\,
(\cD - \cC) \beps(\ddot{\bu})  \,-\, \omega ^{-1} \dot\bze \quad \text{in  $\Omega\times (0, T]$}.
\]
In this way, defining $\cA := \cC^{-1}$ and $\cV:= (\cD - \cC)^{-1}$, we find that  problem \eqref{model} can be stated as follows:
 \begin{equation}\label{split123}
 \begin{array}{rcll}
 \rho\ddot{\bu}-\bdiv(\bze+\bgam) &=& \bF & \quad\text{in $\Omega\times (0, T]$}, 
 \\[1ex]
 \cV \ddot{\bze} + \tfrac{1}{\omega} \cV \dot{\bze} &=& \beps(\ddot{\bu}) & \quad \text{in $\Omega\times (0, T]$}, 
 \\[1ex]
 \cA \ddot{\bgam} &=& \beps(\ddot{\bu}) & \quad \text{in $\Omega\times (0, T]$}.
\end{array}
\end{equation}
The essential boundary condition on $\Gamma_N$ requires the introduction of the closed 
subspace of $\mathbb{H}(\bdiv, \Omega)$ given by 
\[
\mathbb W := \set{\btau\in \mathbb{H}(\bdiv, \Omega); \quad 
	\left\langle\btau\bn,\bv\right\rangle_{\partial\Omega}= 0 
	\quad \text{$\forall\bv\in \bH^{1/2}(\partial\Omega)$,\, $\bv|_{\Gamma_D} = \mathbf{0}$}},
\]
where $\left\langle\cdot,\cdot\right\rangle_{\partial\Omega}$ stands for the duality pairing 
between $\bH^{-1/2}(\partial\Omega)$  and $\bH^{1/2}(\partial\Omega)$ with respect to the $\bL^2(\partial\Omega)$-inner product. In the sequel, we use the compact notations $\kp:=(\bze, \bgam)$ and $\kq := (\btau, \bet)$ to denote elements from $\mathfrak{L}^2(\Omega)$ and set $\kq^+:= \btau + \bet$. We also introduce the $\mathfrak{L}^2(\Omega)$-orthogonal projection $\pi_1\kq:= (\btau, \mathbf 0)$ onto $\mathbb L^2(\Omega)\times \{\mathbf 0\}$.  The energy space corresponding to the variational formulation of \eqref{split123} is given by
\[
 \mathfrak{S}:= \Big\{\kq \in \mathfrak{L}^2(\Omega):\ \kq^+ \in \mathbb W  \Big\}. 
\]
It is endowed with the Hilbertian norm $\norm{\kq}^2_{\mathfrak{S}}:= \norm{\kq}_{0,\Omega}^2+ \norm{\bdiv \kq^+}^2_{0,\Omega}$. We notice that the embeddings $\mathbb W\times \mathbb W  \hookrightarrow \mathfrak{S}   \hookrightarrow \mathfrak{L}^2(\Omega)$ are continuous. On the other hand, we define the space of symmetric tensors with square integrable entries $\mathbb{L}^2_{\text{sym}}(\Omega):= \big\{\btau \in \mathbb{L}^2(\Omega):\ \btau = \btau^{\t}\big\}$, and denote by $\mathbb{Q}:= \big\{\btau \in \mathbb{L}^2(\Omega):\ \btau = -\btau^{\t}\big\}$ its orthogonal complement in $\mathbb{L}^2(\Omega)$. We also  let $\mathfrak L_{\text{sym}}^2(\Omega) := \set{\kq \in \mathfrak L^2(\Omega):\, \kq^+ \in \mathbb{L}^2_{\text{sym}}(\Omega)}$ and consider the closed subspace $\mathfrak{S}_{\text{sym}} := \mathfrak{S}\cap\mathfrak L_{\text{sym}}^2(\Omega)$ of $\mathfrak{S}$. 
\begin{lemma}\label{density}
The embedding $\mathbb W \times \mathbb W  \hookrightarrow \mathfrak S$ is dense.
\end{lemma}
\begin{proof}
	Let $\kp:=(\bze, \bgam)\in \mathfrak S$ be such that
	\begin{equation}\label{step}
		(\kp, \kq)  + \big(\bdiv\kp^+, \bdiv\kq^+\big) = 0 \quad \forall \kq:=(\btau,\bet)\in  \mathbb W\times \mathbb W.
	\end{equation}
	Taking $\btau = \bet = \kp^+\in \mathbb W$ in \eqref{step} we deduce that $\kp^+ := \bze + \bgam=\mathbf 0$ in $\mathbb{H}(\bdiv, \Omega)$. Using this fact in \eqref{step} and testing with $\pi_1\kq:= (\btau, \mathbf 0)$ proves that $\bze = \mathbf 0$ and the result follows. 
\end{proof}

We obtain a variational formulation of \eqref{split123} by considering an arbitrary $\kq =(\btau,\bet) \in \mathfrak S$ and by testing the constitutive laws (second and third rows of \eqref{split123}) with $\btau$ and $\bet$, respectively. Adding the resulting equations,  we obtain 
\begin{equation}\label{const+}
( \cV \ddot{\bze} + \tfrac{1}{\omega} \cV \dot{\bze},\btau) + (\cA \ddot{\bgam},\bet) = (\beps(\ddot{\bu}), \btau + \bet) = \big(\nabla \ddot\bu - \ddot \br, (\btau + \bet)\big),
\end{equation}
where the skew symmetric tensor $\br := \frac{1}{2}\big\{\nabla\bu-(\nabla\bu)^{\t}\big\}$ is the rotation. Next, integrating by parts in the right hand-side of \eqref{const+} and employing the Dirichlet boundary condition $\bu = \mathbf 0$ on $\Gamma_D \times [0,T]$ together with the fact that $\kq\in \mathfrak S$, we obtain
\begin{equation}\label{var0}
( \cV \ddot{\bze} + \tfrac{1}{\omega} \cV \dot{\bze},\btau) + (\cA \ddot{\bgam},\bet) =   
- \, \big(\ddot{\bu}, \bdiv(\btau + \bet)\big) - (\ddot \br,\btau + \bet).
\end{equation}
Now, we use the  motion equation (first row of \eqref{split123}) to eliminate the displacement field from \eqref{var0}. Indeed, substituting back $\ddot{\bu} = \rho^{-1}(\bF + \bdiv\bsig)$ into \eqref{var0} we  end up with 
\begin{equation}\label{var1}
A\big(\ddot\kp,\kq\big) +  A\big(\tfrac{1}{\omega}\pi_1\dot\kp, \kq\big) +  (\ddot \br,\kq^+)  + \big(\bdiv\kp^+ , \bdiv\kq^+ \big)_\rho  = - \big(\bF , \bdiv\kq^+ \big)_\rho,
\end{equation}
where $(\bu, \bv)_{\rho} := (\tfrac{1}{\rho}\bu, \bv)$ for all $\bu, \bv \in \mathbf L^2(\Omega)$, and 
\[
A\big( \kp,\kq\big) := (\cV \bze, \btau) + (\cA\bgam, \bet), \quad \forall \kp=(\bze,\bgam),\, \kq=(\btau,\bet) \in \mathfrak{L}^2(\Omega).
\]

As a consequence of our hypotheses on $\cC$ and $\cD$, the bilinear form $A$ is symmetric, 
bounded and coercive, which means that there exist positive constants $M$ and $\alpha$, depending only on $\cC$ and $\cD$, such that
\begin{equation}\label{contA}
\big| A( \kp,\kq) \big| \leq M \| \kp \|_{0,\Omega} 
\|\kq \|_{0,\Omega} \qquad \forall \, \kp, \kq \in \mathfrak{L}^2(\Omega),
\end{equation}
and 
\begin{equation}\label{ellipA}
 A( \kq,\kq) \geq \alpha  
\|\kq \|^2_{0,\Omega} \qquad \forall  \,\kq \in \mathfrak{L}^2(\Omega).
\end{equation}

Taking into account these notations and \eqref{var1}, we deduce that, given $\bF\in \L^{1}(\bL^2(\Omega))$, 
the mixed variational formulation of our problem reads as follows:  Find  $\kp\in \L^{\infty}(\mathfrak{S}) \cap \W^{1,\infty}(\mathfrak{L}^2(\Omega))$ and $\br \in \W^{1,\infty}(\mathbb{Q})$  such that
\begin{align}\label{varFormR1-varFormR2}
\begin{array}{rcll}
\frac{\text{d}}{\text{d}t}  \Big\{ A\big(\dot\kp +  \tfrac{1}{\omega} \pi_1\kp, \kq\big)  +(\dot\br, \kq^+)\Big\}+ \big(\bdiv\kp^+ , \bdiv\kq^+\big)_\rho
&= &-\big( \bF, \bdiv\kq^+\big)_\rho, & \forall \kq \in \mathfrak{S}
\\[1ex]
 (\bs,\kp^+) &=& 0, & \forall \bs\in \mathbb{Q},
 \end{array}
\end{align} 
and such that the initial conditions 
\begin{align}\label{initial-R1-R2}
\begin{split}
\kp(0) &= \kp_0 = (\bze_0, \bgam_0) \in \mathfrak{S}_{\text{sym}},
\quad  \dot\kp(0) = \kp_1 = (\bze_1, \bgam_1) \in \mathfrak{L}_{\text{sym}}^2(\Omega),
\\[1ex]
\br(0)&= \br_0\in \mathbb{Q},\qquad  \dot\br(0) = \br_1 \in \mathbb{Q} \,,
\end{split}
\end{align}
are satisfied with
\begin{align}\label{initcompatible}
	\begin{split}
		\bgam_0 &:= \cC\beps(\bu_0), \quad \bgam_1 :=  \cC\beps(\bu_1), \quad \bze_0 := \bsig_0 - \bgam_0,\quad \bze_1 := \cD \beps(\bu_1) - \bgam_1 - \omega^{-1} \bze_0
\\[1ex]
\br_0 &:= \nabla \bu_0 -  \beps(\bu_0),\quad\text{and}\quad \br_1 := \nabla \bu_1 -  \beps(\bu_1).
	\end{split}
\end{align}  

\section{The Arnold-Falk-Winther mixed finite element}\label{section4}

We consider  shape regular meshes $\mathcal{T}_h$ that subdivide the domain $\bar \Omega$ into  triangles/tetrahedra $K$ of diameter $h_K$. The parameter $h:= \max_{K\in \cT_h} \{h_K\}$ represents the mesh size of $\cT_h$.  In what follows, we assume that $\mathcal{T}_h$ is compatible with the partition $\bar\Omega = \cup_{j= 1}^J \bar{\Omega}_j$, i.e., 
\[
\bar{\Omega}_j = \cup \set{K\in\cT_h; \quad  K\subset \bar{\Omega}_j} \quad \forall j=1,\cdots, J.
\]
Hereafter, given an integer $m\geq 0$, the space of piecewise polynomial functions of degree at most $m$ relatively to $\cT_h$ is denoted by  
\[
 \cP_m(\cT_h) :=\set{ v\in \L^2(\Omega); \quad v|_K \in \cP_m(K),\quad \forall K\in \cT_h }. 
\]
For $k\geq 1$, the finite element spaces
\[
 \mathbb{W}_h := \cP_k(\cT_h)^{ d\times d} \cap \mathbb W, \qquad \mathbb{Q}_h := \cP_{k-1}(\cT_h)^{d\times d} \cap \mathbb{Q} \qquad \text{and} \qquad \mathbf{U}_h := \cP_{k-1}(\cT_h)^{d}
\]
correspond to the Arnold-Falk-Winther family introduced in \cite{ArnoldFalkWinther} for the steady elasticity problem. 

We notice that due to the  embedding $\mathbb W\times \{\mathbf 0\} \hookrightarrow \mathfrak{S}$, the inf-sup condition satisfied (cf. \cite{abd, BoffiBrezziFortin}) by the bilinear form $\big(\btau, (\bs, \bv) \big)\mapsto (\bs,\btau) + (\bv,\bdiv \tau)$ for the pair $\{\mathbb W,   \mathbb{Q} \times\bL^2(\Omega)\}$ (cf. \eqref{inf-sup-S-Q}) implies immediately that  there exists $\beta>0$ such that 
\begin{equation}\label{InfSup}
\sup_{\kq \in \mathfrak S} 
\frac{(\bs,\kq^+) \,+\, \big(\bv,\bdiv \kq^+\big)}{\norm{\kq}_{\mathfrak S}}
\,\ge\, \beta \Big\{ \norm{\bs}_{0,\Omega} + \norm{\bv}_{0,\Omega} \Big\},\quad \forall (\bs,\bv) \in \mathbb{Q} \times\bL^2(\Omega).
\end{equation}
Similarly, it is shown in \cite{ArnoldFalkWinther} that $\big(\btau, (\bs, \bv) \big)\mapsto (\bs,\btau) + (\bv,\bdiv \tau)$ satisfies a uniform inf-sup condition for the pair $\{\mathbb W_h,   \mathbb{Q}_h \times\mathbf U_h\}$. Therefore, if we let $\mathfrak S_h:= \mathbb{W}_h \times \mathbb{W}_h \subset \mathfrak{S}$, the embedding $\mathbb W_h\times \{\mathbf 0\} \hookrightarrow \mathfrak{S}_h $ implies the existence of $\beta^*>0$, independent of $h$, such that 
\begin{equation}\label{discInfSup}
\sup_{\kq \in \mathfrak S_h} 
\frac{(\bs,\kq^+) \,+\, \big(\bv,\bdiv \kq^+\big)}{\norm{\kq}_{\mathfrak S}}
\,\ge\, \beta^* \Big\{ \norm{\bs}_{0,\Omega} + \norm{\bv}_{0,\Omega} \Big\},\quad \forall (\bs,\bv) \in \mathbb{Q}_h \times\mathbf{U}_h.
\end{equation} 

We recall that the tensorial version $\Pi_h:  \mathbb H^1(\Omega)\cap \mathbb W \to \mathbb{W}_h$  of the 
BDM-interpolation operator satisfies the following classical error estimate, see \cite[Proposition 2.5.4]{BoffiBrezziFortinBook},
\begin{equation}\label{asymp}
 \norm{\btau - \Pi_h \btau}_{0,\Omega} \leq C h^{m} \norm{\btau}_{m,\Omega} \qquad \forall \btau \in \mathbb H^{m}(\Omega) \quad \text{with $1 \leq m \leq k+1$}.
\end{equation}
Moreover, if $\bdiv \btau \in \bH^k(\Omega)$, the commuting diagram property implies that   
\begin{equation}\label{asympDiv}
 \norm{\bdiv (\btau - \Pi_h \btau) }_{0,\Omega} = \norm{\bdiv \btau -  U_h \bdiv \btau }_{0,\Omega} 
 \leq C h^{m} \norm{\bdiv\btau}_{m,\Omega}\quad \text{for $0 \leq m \leq k$},
\end{equation}   
where $U_h$ is the orthogonal projection from $(\bL^2(\Omega), \norm{\cdot}_{0,\Omega})$ onto $\mathbf{U}_h$. 

In order to facilitate our forthcoming analysis, we now introduce an auxiliary operator $\varXi_h:\ \mathfrak{S} \to \mathfrak{S}_h$, whose construction is adapted from  \cite[Lemma 2.1]{Arnold-Lee}. More precisely, for each $\kp \in \mathfrak{S}$ we define $\varXi_h\kp := \widetilde\kp_h$, where $(\widetilde\kp_h, \widetilde\br_h, \widetilde\bu_h)\in \mathfrak{S}_h \times \mathbb{Q}_h \times \mathbf U_h$ is the solution of 
\begin{align}\label{Xih}
\begin{split}
\big(\widetilde \kp_h, \widetilde \kq\big) + (\widetilde\br_h, \kq^+) 
+ ( \widetilde\bu_h, \bdiv \kq^+ )_\rho  &= \big( \kp , \kq\big) \quad \forall \, \kq \in \mathfrak{S}_h,\\[1ex]
(\bs, \widetilde\kp^+_h) + (\bv,\bdiv \widetilde \kp^+_h)_\rho &= (\bs, \kp^+) + \big(\bv, \bdiv\kp^+ \big)_\rho \quad \forall \,(\bs,\bv)\in \mathbb{Q}_h\times \mathbf U_h.
\end{split}
\end{align}

Let $\mathfrak{S}_{\text{sym},h} := \big\{ \kq\in \mathfrak{S}_h:\ (\bs,\kq^+) = 0 \quad \forall \,\bs \in \mathbb{Q}_h \big\}$ be the kernel of the bilinear form $\mathfrak{S}_h\times \mathbb{Q}_h \ni (\kq, \bs) \mapsto (\bs , \kq^+)$. It is important to realize that $\mathfrak{S}_{\text{sym},h}$ is not  a subspace of $\mathfrak{S}_{\text{sym}}$.  We denote by $\mathfrak{K}$ and $\mathfrak{K}_{h}$ the kernels of the bilinear form $\mathfrak S \times (\mathbb Q \times \mathbf L^2(\Omega)) \ni \big(\kq, (\bs, \bv) \big)\mapsto (\bs,\kq^+) + (\bv,\bdiv \kq^+)$ and its restriction to $\mathfrak S_h \times (\mathbb Q_h\times \mathbf U_h)$, respectively, that is
\[
\mathfrak{K} := \big\{ \kq\in \mathfrak{S}_{\text{sym}}: \  \bdiv \kq^+ = \mathbf 0 \big\},
\quad
\text{and}
\quad
\mathfrak{K}_{h} := \big\{ \kq\in \mathfrak{S}_{\text{sym,h}}:\  \bdiv \kq^+ = \mathbf 0 \big\}\,.
\]
Then, the fact that $(\kq,\kq) = \norm{\kq}_{0,\Omega}^2 = \norm{\kq}^2_{\mathfrak S}$ for all $\kq \in \mathfrak{K}$, the inf-sup condition \eqref{InfSup}, and the Babu\v{s}ka-Brezzi theory, guarantee that the continuous counterpart of problem \eqref{Xih} is well-posed, whose unique solution is easily seen to be $(\kp, \mathbf 0, \mathbf 0) \in \mathfrak S \times \mathbb Q \times \mathbf L^2(\Omega)$. In turn, noting that there certainly holds $(\kq,\kq) = \norm{\kq}^2_{\mathfrak S}$ for all $\kq \in \mathfrak{K}_h$ as well, and employing now the inf-sup condition \eqref{discInfSup} and the discrete Babu\v{s}ka-Brezzi theory, we deduce that problem \eqref{Xih} is well-posed (uniformly in $h$). Moreover, C\'ea's estimate between $(\kp, \mathbf 0, \mathbf 0)$ and $(\widetilde\kp_h, \widetilde\br_h, \widetilde\bu_h)$ implies the following approximation property for $\varXi_h$:
\begin{equation}\label{XihStab}
	\norm{\kp - \varXi_h\kp}_{\mathfrak S} \lesssim
\inf_{\kq_h\in \mathfrak{S}_h}\norm{\kp-\kq_h}_{\mathfrak{S}} \quad \forall \kp\in \mathfrak S.
\end{equation}
Moreover,  taking into account  the compatibility of $\mathcal{T}_h$ with the partition $\bar\Omega = \cup_{j= 1}^J \bar{\Omega}_j$, it turns out that  $\varXi_h$ satisfies by construction the commuting diagram property 
\begin{equation}\label{div0}
\tfrac{1}{\rho}\bdiv (\varXi_h\kp)^+ = U_h (\tfrac{1}{\rho} \bdiv\kp^+), \quad \forall \kp \in \mathfrak S.
\end{equation}

It is important to stress in advance that the analysis that will follow also holds true for  most other known mixed finite elements \cite{abd,  CGG, GG, stenberg} for the steady elasticity problem  with reduced symmetry. However, for the sake of brevity we are restricting our choice of finite element examples to the Arnold-Falk-Winther  family \cite{ArnoldFalkWinther}.

\section{Well-posedness of the continuous problem}\label{section5}

We consider the following semi-discrete counterpart of  \eqref{varFormR1-varFormR2}-\eqref{initial-R1-R2}: Find $\kp_h\in \cC^1(\mathfrak{S}_h)$ and $\br_h\in \cC^1(\mathbb{Q}_h)$ solving 
\begin{equation}\label{varFormR1-varFormR2-h}
\begin{array}{rcll}
\dfrac{\text{d}}{\text{d}t}  \Big\{ A\big( \dot\kp_h + \tfrac{1}{\omega}\pi_1\kp_h , \kq\big)  + (\dot{\br}_h,\kq^+)  \Big\}
+ \big( \bdiv\kp_h^+,  \bdiv \kq^+ \big)_\rho \hspace{-3pt}& = &\hspace{-3pt} -\big(\bF,  \bdiv \kq^+ \big)_\rho, &
\hspace{-5pt}\forall \kq\in \mathfrak{S}_h
\\[1ex]
 (\bs, \kp_h^+) \hspace{-3pt}&=&\hspace{-3pt} 0, &\hspace{-5pt} \forall \bs\in \mathbb{Q}_h. 
 \end{array}
 \end{equation}
We impose to problem \eqref{varFormR1-varFormR2-h} the initial conditions 
\begin{equation}\label{initial-R1-R2-h}
\kp_h(0) = \kp_{0,h} \,,
\quad  \dot{\kp}_h(0) = \kp_{1,h}, \quad \br_h(0) = Q_h\br_{0}, \quad \dot\br_h(0)= Q_h\br_{1},
\end{equation}
where $\kp_{0,h}$ is the $\mathfrak{S}$-orthogonal projection of $\kp_0$ onto $\mathfrak{S}_{\text{sym},h}$, $\kp_{1,h}$ is the $\mathfrak L^2(\Omega)$-orthogonal projection of $\kp_1$ onto $\mathfrak{S}_{\text{sym},h}$ and $Q_h$ is 
the $\bbL^2(\Omega)$-orthogonal projector  onto $\mathbb{Q}_h$. 
It is clear that the component $\kp_h$ of problem \eqref{varFormR1-varFormR2-h}-\eqref{initial-R1-R2-h} 
solves the following reduced formulation: Find $\kp_h\in \cC^1(\mathfrak{S}_{\text{sym},h})$  
satisfying $\kp_h(0) = \kp_{0,h}$, $\dot\kp_h(0) = \kp_{1,h}$, and such that 
\begin{equation}\label{varForm-h}
\dfrac{\text{d}}{\text{d}t}  A\big( \dot\kp_h + \tfrac{1}{\omega}\pi_1 \kp_h,\kq\big) +  \big( \bdiv \kp_h^+,  \bdiv \kq^+ \big)_\rho = - \big( \bF,  \bdiv \kq^+ \big)_\rho,\quad \forall \kq\in \mathfrak{S}_{\text{sym},h}.
\end{equation}
In this regard, the unique solvability of \eqref{varForm-h} is obtained by writing the problem in the form of a first order system of ODEs and applying the classical Cauchy-Lipschitz-Picard theorem. 

In the following result, we  obtain stability estimates for the solution $\kp_h(t)$ in terms of the energy functional $\mathcal{E}:\, W^{1,\infty}(\mathfrak{S}) \to \L^\infty((0,T))$  defined by
\begin{equation}\label{defE}
\mathcal{E}\big(\kq\big)(t):= \frac{1}{2} A\big(\dot\kq(t),\dot\kq(t)\big) + \frac{1}{2}  
\big( \bdiv \kq^+(t), \bdiv \kq^+(t)\big)_\rho.
\end{equation}

\begin{theorem}
Assume that $\bF \in \W^{1,1}(\bL^2(\Omega))$. Then,  problem \eqref{varFormR1-varFormR2-h}-\eqref{initial-R1-R2-h} admits a unique solution  satisfying 
\begin{equation}\label{AprioriBound_h}
\max_{t\in [0, T]}\mathcal{E}(\kp_h)^{1/2}(t)  + \max_{t\in [0, T]}\norm{\dot{\br}_h}_{0,\Omega} \lesssim \norm{\bF}_{\W^{1,1}(\bL^2(\Omega))} + \norm{\kp_0}_{\mathfrak{S}} + \norm{\kp_1}_{0,\Omega}.
\end{equation}
\end{theorem}
\begin{proof}
We take  $\kq = \dot\kp_h$ in \eqref{varForm-h} and integrate the resulting identity over $(0, t)$ to obtain
\[ 
\mathcal{E}\big(\kp_h\big)(t)  +  \int_0^t A\big(\tfrac{1}{\omega}\pi_1\dot\kp_h(s),\pi_1\dot\kp_h(s)\big)\, \text{d}s = \mathcal{E}\big(\kp_h\big)(0)  - \int_0^t \big( \bF(s),  \bdiv \dot\kp_h^+(s) \big)_\rho \, \text{d}s.
\]
Next, integrating by parts the second term on the right-hand side, and using that the second term 
on the left hand side is non-negative, we find
\begin{equation}\label{estimN0}
\begin{array}{c}
\disp
\mathcal{E}\big(\kp_h\big)(t)  \leq 
\int_0^t \big( \dot{\bF}(s), \bdiv\kp_h^+(s)\big)_\rho \, \text{d}s -  \big( \bF(t), \bdiv \kp_h^+(t)\big)_\rho  
+ \Big( \bF(0), \bdiv \kp_{0,h}^+\Big)_\rho + \mathcal{E}\big(\kp_h\big)(0)\,.
\end{array}
\end{equation}
By virtue of the Cauchy-Schwarz inequality and the Sobolev embedding 
$\W^{1,1}(\bL^2(\Omega)) \hookrightarrow \cC^0(\bL^2(\Omega))$ (see \cite[Lemma 7.1]{Roubicek}),
it follows from \eqref{estimN0} that 
\begin{align}\label{otra-ineq-E}
\mathcal{E}\big(\kp_h\big)(t)   \lesssim 
 \|\bF\|_{\W^{1,1}(\bL^2(\Omega))} \max_{t\in [0, T]}\mathcal{E}\big(\kp_h\big)^{1/2}(t) + \mathcal{E}\big(\kp_h\big)(0). 
\end{align}
Moreover,  it is clear from the definition of $\mathcal{E}$ that $\mathcal{E}\big(\kp_h\big)(0) \lesssim \norm{\kp_{0,h}}^2_{\mathfrak S} + \norm{\kp_{1,h}}^2_{0,\Omega}$, which together with \eqref{otra-ineq-E}, leads to
\begin{equation}\label{interm}
\max_{t\in [0, T]}\mathcal{E}\big(\kp_h\big)^{1/2}(t) \lesssim \|f\|_{\W^{1,1}(\bL^2(\Omega))} 
+ \norm{\kp_{0,h}}_{\mathfrak{S}}+ \norm{\kp_{1,h}}_{0,\Omega}\,.
\end{equation}
We turn now to prove the existence of a unique Lagrange multiplier $\br_h$. To this end, we let $\mathcal{G}_h(t)\in \cC^1(\mathfrak{S}_h')$ be given as follows (in terms of $\br_{0,h}$, $\br_{1,h}$ and the solution $\kp_h$ of \eqref{varForm-h}):
\begin{align*}
\mathcal{G}_h(t)\big(\kq\big)  &:= A\big(\kp_h(t) +\tfrac{1}{\omega} \int_0^t \pi_1\kp_h(s)\,\text{d}s, \kq\big ) + \int_0^t\big(\int_0^s \big(\bF(z) + \bdiv \kp_h^+(z),\,  \bdiv \kq^+ \big)_\rho \, \text{d}z \big)\text{d}s
 \\[1ex]
&\qquad - t \Big\{A\big( \kp_{1,h} + \tfrac{1}{\omega} \pi_1 \kp_{0,h} , \kq\big) +   (\br_{1,h}, \kq^+)  \Big\} - \Big\{ A\big(\kp_{0,h}, \kq\big) + (\br_{0,h},\kq^+)\Big\}.
\end{align*}
Integrating   \eqref{varForm-h} twice with respect to time we deduce that the functional $\mathcal{G}_h(t)\in \mathfrak{S}_h'$ vanishes identically on the kernel $\mathfrak{S}_{\text{sym},h}$ of the bilinear form $\mathfrak{S}_h \times \mathbb{Q}_h \ni \big(\kq, \br\big) \mapsto (\br,\kq^+)$. Therefore, the discrete inf-sup condition \eqref{discInfSup} implies the existence of  a unique $\br_h\in \cC^1(\mathbb{Q}_h)$ such that 
\begin{equation}\label{bRh}
 ( \br_h(t), \kq^+) = -\mathcal{G}_h(t)\big(\kq\big), \quad  
\forall\, t\in [0,T]\,, \quad \forall \, \kq \in \mathfrak{S}_h.
\end{equation}
Differentiating twice the last identity in the sense of distributions we deduce that  
$\big(\kp_h, \br_h\big)$ solves \eqref{varFormR1-varFormR2-h}. Moreover, evaluating 
\eqref{bRh} and its time derivative at $t=0$,  and using again the  discrete inf-sup condition 
\eqref{discInfSup}, we deduce that the initial conditions $\br_h(0) = \br_{0,h}$ and 
$\dot{\br}_h(0) = \br_{1,h}$ are fulfilled. Finally, using \eqref{discInfSup} once again we 
obtain from the Cauchy-Schwarz inequality the estimate 
\begin{align}\label{comb11}
\begin{array}{c}
\disp
\beta^*\,\norm{\dot{\br}_h(t)}_{0,\Omega} \leq 
\sup_{\kq\in \mathfrak{S}_h} \dfrac{ (\dot{\br}_h(t),\kq^+)}{\norm{\kq}_{\mathbb W}} = 
\sup_{\kq\in \mathfrak{S}_h} \frac{\dot{\mathcal{G}}_h(t)\big(\kq\big)} {\norm{\kq}_{\mathbb W}} 
\lesssim \max_{t\in[0,T]}\mathcal{E}\big(\kp_h\big)^{1/2}(t) \\[2ex]
\disp 
+ \,\, \norm{\bF}_{\W^{1,1}(\bL^2(\Omega))} + \|\kp_{0,h}\|_{0,\Omega} + \norm{\kp_{1,h}}_{0,\Omega} +  \|\br_{1,h}\|_{0,\Omega}\,.
\end{array}
\end{align}
The stability result  \eqref{AprioriBound_h} is then a consequence of  \eqref{interm},   \eqref{comb11}, 
and the definition of the discrete initial data (cf. \eqref{initial-R1-R2-h}).
\end{proof}

The a priori estimate \eqref{AprioriBound_h} permits to employ the classical Galerkin 
procedure (cf. \cite{Evans, Renardy}) to prove that the continuous problem 
\eqref{varFormR1-varFormR2}-\eqref{initial-R1-R2} admits at least a solution. 
To this end, we need to assume that for each pair $(\kq,\br) \in \mathfrak S \times \mathbb{Q}$ there holds
\begin{equation}\label{appProp}
\lim_{h\to 0} \Big\{\,\inf_{\kq_h\in \mathfrak{S}_h} \norm{\kq - \kq_h}_{\mathfrak S} + \inf_{\bs_h\in \mathbb{Q}_h} \norm{\br - \bs_h}_{0,\Omega} \,\Big\} =0.
\end{equation}
This approximation property holds true if $\Gamma_D = \partial \Omega$ or $\Gamma_N = \partial \Omega$ because of \eqref{asymp}, \eqref{asympDiv} and of the density of smooth functions in $\mathbb W\times \mathbb W$ (cf. \cite{girault2012finite}), and therefore also in $\mathfrak S$ thanks to Lemma~\ref{density}. To our knowledge, such density results are not known when $\Gamma_N\subsetneq \partial \Omega$. Nevertheless, as $\mathfrak S\times \mathbb Q$ is a separable Hilbert space,  it is possible to avoid this technical inconvenient by performing the Galerkin dimension reduction of \eqref{varFormR1-varFormR2}-\eqref{initial-R1-R2} by means of a countable set of linearly independent elements whose linear span is dense in $\mathfrak S\times \mathbb Q$, cf. \cite{ggm-JSC-2017} for an example. Here, for the sake of brevity, we  illustrate the Galerkin procedure directly in terms of the mixed finite element that defines our scheme.   
\begin{theorem}\label{theorem-R1-R2}
Assume that $\bF \in \W^{1,1}(\bL^2(\Omega))$. Then,  problem \eqref{varFormR1-varFormR2}-\eqref{initial-R1-R2} admits a unique solution. Moreover, there holds
\begin{align}\label{AprioriBoundR}
\begin{split}
\disp
\operatorname*{sup~ess}_{t\in [0,T]}\|\kp(t)\|_{\mathfrak{S}} &+ \|\kp(t)\|_{\W^{1,\infty}(\mathfrak L^2(\Omega))}  + \norm{\br}_{\W^{1,\infty}(\mathbb{L}^2(\Omega))} \lesssim  \norm{\bF}_{\W^{1,1}(\bL^2(\Omega))}
\\
\disp
& + \|\kp_0\|_{\mathfrak{S}} + \|\kp_1\|_{0,\Omega} 
+ \|\br_0\|_{0,\Omega} + \|\br_1\|_{0,\Omega}.
\end{split}
\end{align}
\end{theorem}
\begin{proof}
We first observe from the definition of $\mathcal E$ and  \eqref{ellipA} that 
\begin{equation}\label{eq-extra-1}
\|\dot\kq(t)\|^2_{0,\Omega} +
\|\bdiv\kq^+(t)\|^2_{0,\Omega} \lesssim
\mathcal{E}\big(\kq\big)(t),\quad \forall \kq \in \L^{\infty}(\mathfrak{S}) \cap \W^{1,\infty}(\mathfrak{L}^2(\Omega)).
\end{equation}
Hence, it follows from \eqref{AprioriBound_h} that  $\big\{\dot\kp_h\big\}_{h}$, $\big\{\kp_h\big\}_{h}$ and $\{\br_h\}_{h}$ are uniformly bounded in  $\L^\infty(\mathfrak{L}^2(\Omega))$, $\L^\infty(\mathfrak{S})$ and $\W^{1,\infty}(\mathbb Q)$, respectively.  We can then extract weak$^*$ convergent subsequences (also denoted $\big\{\kp_h\big\}_{h}$ and $\{\br_h\}_h$) with limits $\kp \in \L^\infty(\mathfrak{S})\cap \W^{1,\infty}(\mathfrak{L}^2(\Omega))$ and $\br \in \W^{1,\infty}(\mathbb Q)$, respectively. We deduce immediately from the second equation of \eqref{varFormR1-varFormR2-h} and the density of $\mathbb Q_h$ in $\mathbb Q$ that $\kp^+\in \mathbb{L}^2_{\text{sym}}(\Omega)$. Moreover, multiplying the first equation of \eqref{varFormR1-varFormR2-h} by a function $\psi\in \cC^1([0, T])$ satisfying $\psi(T) = 0$,
and integrating by parts with respect to $t \in [0,T]$, yield 
\begin{align}\label{subsequence}
\begin{split}
 -\int_0^T \Big\{ 
 A\big( \dot\kp_h + \tfrac{1}{\omega} \pi_1\kp_h, \kq \big)  &+ (\dot{\br}_h, \kq^+) 
 \Big\} \dot{\psi}(t)\,  \text{d}t + \int_0^T \big(\bdiv \kp_h^+ + \bF(t), \bdiv\kq^+ \big)_\rho \psi(t)\, \text{d}t
 \\[1ex] 
 & =\psi(0) \Big\{ A\big( \kp_{1,h} + \tfrac{1}{\omega}\pi_1\kp_{0,h}, \kq\big) + (\br_{1,h}, \kq^+) \Big\},\quad \forall \kq \in \mathfrak{S}_h.
\end{split}	
\end{align}
Passing  to the limit in the last equation and using \eqref{appProp} shows that $\kp$ and $\br$ satisfy
\begin{align}\label{limit}
\begin{split}
-\int_0^T \Big\{ A\big( \dot\kp + \tfrac{1}{\omega} \pi_1\kp, \kq \big)  &+ (\dot{\br}, \kq^+) \Big\} \dot{\psi}(t)\,  \text{d}t + \int_0^T \big(\bdiv \kp^+ + \bF(t), \bdiv\kq^+ \big)_\rho \psi(t)\, \text{d}t
 \\[1ex] 
 & =\psi(0) \Big\{ A\big( \kp_{1} + \tfrac{1}{\omega}\pi_1\kp_{0}, \kq\big) + (\br_{1}, \kq^+) \Big\},
\end{split}	
\end{align}
for all $\kq \in \mathfrak{S}$ and for all $\psi\in \cC^1([0, T])$ such that $\psi(T) = 0$. This proves that $(\kp,\br)$ solves the first equation of \eqref{varFormR1-varFormR2} provided the time derivative is understood in the sense of distributions on $(0, T)$. In addition, we notice that $\kp_h$ also converges weakly to $\kp$ in $\H^{1}(\mathfrak L^2(\Omega))$. Hence $\kp_h(0)$ converges weakly to $\kp(0)$ in $\mathfrak L^2(\Omega)$, and since $\kp_h(0)= \kp_{0,h}$ also converges strongly to $\kp_0$ in $\mathfrak L^2(\Omega)$, we conclude that $\kp(0)=\kp_0$. Similarly, the sequence $\{\br_h(0)\}_h= \{\br_{0,h}\}_h$ converges weakly to $\br(0)$ in $\mathbb{Q}$ and strongly to $\br_0$ in $\mathbb{Q}$, which gives $\br(0) = \br_{0}$. To obtain the remaining initial conditions we take $\psi$ in \eqref{limit} such that $\psi(0) = 0$ and integrate the first term backwardly with respect to $t$ to get 
\begin{equation}\label{strongT}
\frac{\text{d}}{\text{d}t} \Big\{ A\big( \dot\kp + \tfrac{1}{\omega}\pi_1\kp, \kq\big) + (\dot{\br}, \kq^+) \Big\}= - \big(\bdiv\kp^+ + \bF(t),\, \bdiv\kq^+\big)_\rho,\quad \forall \kq \in \mathfrak{S}.
\end{equation}
It follows that $ t \mapsto A\big( \dot\kp + \tfrac{1}{\omega}\pi_1\kp, \kq\big) + (\dot{\br}, \kq^+) $ 
belongs to $\W^{1,1}(\mathfrak{S}') \hookrightarrow \cC^0(\mathfrak{S}')$, 
and we can test \eqref{strongT} with a function $\psi\in \cC^1([0, T])$ satisfying $\psi(T) = 0$ to get
\begin{align}\label{limit0}
\begin{split}
\disp
 -\int_0^T \Big\{ \Big( A\big( \dot\kp + \tfrac{1}{\omega}\pi_1\kp, \kq\big) + (\dot{\br},\kq^+)\Big) \Big\}\dot{\psi}(t)\, \text{d}t 
 +  \int_0^T \big(\bdiv \kp^+(t) + \bF(t),\bdiv\kq^+\big)_\rho \psi(t)\,  \text{d}t  
 \\[1ex]
=  \psi(0)\Big( A\big( \dot\kp(0) + \tfrac{1}{\omega}\pi_1\kp(0), \kq\big) + (\dot{\br}(0), \kq^+)\Big)
\end{split}
\end{align}
for all $\kq \in \mathfrak{S}$. Comparing \eqref{limit} with \eqref{limit0} we deduce that $\dot\kp(0) = \kp_1$  and $\dot{\br}(0) = \br_1$. Finally, the stability estimate \eqref{AprioriBoundR} is obtained by taking the limit $h\to 0$ in \eqref{AprioriBound_h}. 
\end{proof}

The usual strategy (cf. \cite{Evans} or \cite[Section 11.2]{Renardy}) providing uniqueness for second-order hyperbolic evolution problems can be applied here as follows. 
\begin{lemma}\label{uniqueness}
The solution of problem \eqref{varFormR1-varFormR2}-\eqref{initial-R1-R2} is unique.
\end{lemma}
\begin{proof}
Let $(\kp, \br)$ be the solution of \eqref{varFormR1-varFormR2}-\eqref{initial-R1-R2} with vanishing source term  and initial conditions. Proceeding as in \cite{Evans, Renardy}, we introduce for each fixed $s\in (0, T)$ the function 
\[
\W^{1, \infty}(\mathfrak{S}_{\text{sym}}) \ni
\mathfrak{y}(t):= \begin{cases}
            -\int_t^s \kp(u)\,  \text{d}u & t < s,\\
            (\mathbf 0, \mathbf 0)\in \mathcal{L}^2(\Omega) & t \geq s .
           \end{cases}
\]
It follows from the first row of \eqref{varFormR1-varFormR2} (with $\bF=\mathbf 0$) that $ t \mapsto A\big( \ddot\kp , \kq) = - A\big(\tfrac{1}{\omega}\pi_1\dot\kp, \kq\big) - \big(\bdiv\kp^+, \bdiv\kq^+\big)_\rho $ is a continuous linear on  $\mathfrak{S}_{\text{sym}}$. Hence, we can take $\kq = \mathfrak{y}$ in \eqref{varFormR1-varFormR2} and integrate the resulting equation over $[0, T]$ to obtain 
\[
\int_0^T A\big(  \ddot\kp + \tfrac{1}{\omega} \pi_1 \dot\kp, \mathfrak{y} \big)\, \text{d}t  + \int_0^T \big(\bdiv \kp^+ , \bdiv\mathfrak{y}^+ \big)_\rho\, \text{d}t = 0.
\]
As the bilinear form $A$ is symmetric, we can integrate by parts in time to get
\[
-\int_0^T A\big( \dot\kp + \tfrac{1}{\omega} \pi_1 \kp, \dot{\mathfrak{y}} \big) \, 
\text{d}t  + 
\int_0^T \big( \bdiv\kp^+, \bdiv\mathfrak{y}^+  \big)_\rho \, \text{d}t = 0,
\]
where we took into account that $\kp(0) = \mathfrak{y}(T) = (\mathbf 0, \mathbf 0)$.  Using the fact that  $\dot{\mathfrak y }(t) = - \kp(t)$ for $0 \leq t < s$  in the last equation gives 
\[
 \frac{1}{2}\int_0^s \frac{\text{d}}{\text{d}t} \Big\{  (\bdiv\mathfrak{y}^+, \bdiv\mathfrak{y}^+)_{\rho} -  A\big( \kp, \kp \big)\Big\} \text{d}t -  \int_0^s A\big(\tfrac{1}{\omega}\pi_1 \kp, \pi_1\kp \big)\, \text{d}t= 0,
\]
which can equivalently be written 
\[
 \big(\bdiv\mathfrak{y}^+(0), \bdiv\mathfrak{y}^+(0)\big)_{\rho} + A\big( \kp(s), \kp(s) \big) + \int_0^s A\big(\tfrac{1}{\omega}\pi_1 \kp, \pi_1\kp \big)\, \text{d}t= 0,
\]
and the coerciveness of $A$  implies that $\kp =\mathbf 0$ in $\mathfrak{L}^2(\Omega)$.  Finally, we deduce from the first equation of \eqref{varFormR1-varFormR2} and the homogeneous initial conditions on $\br$ that $(\br, \kq^+) = 0$ for all $\kq\in \mathfrak{S}$, which, thanks to the continuous inf-sup condition \eqref{InfSup}, implies that $\br = \mathbf 0$.
\end{proof}

We end this section with a couple of important remarks. Indeed, following \cite[Section 11.2.4]{Renardy}, 
one can show that the solution $(\kp, \br)$ of \eqref{varFormR1-varFormR2}-\eqref{initial-R1-R2} satisfies   
$\kp\in \cC^0(\mathfrak S)\cap \cC^{1}(\mathfrak{L}_{\text{sym}}^2(\Omega))$ and $\br \in \cC^{1}(\mathbb{Q})$.
In turn, in the case of zero source term $\bf$, the identity  
	\[
	\dfrac{\text{d}}{\text{d}t} {\mathcal{E}}\big(\kp_h\big)(t) = - A\big(\tfrac{1}{\omega}\pi_1 \dot\kp, \pi_1\dot\kp \big) \leq 0\,,
	\]
which is obtained by taking $\kq = \dot\kp$ in \eqref{varFormR1-varFormR2}, proves that the viscoelastic material 
does dissipate energy.  

\section{Convergence analysis of the semi-discrete problem}\label{section6}

The elliptic projector $\varXi_h:\mathfrak S\to \mathfrak S_h$ introduced in Section~\ref{section4} will allow us to use standard techniques of error analysis for our scheme. From now on we assume that $\kp_1\in \mathfrak{S}_{\text{sym}}$ and we consider  a solution $(\kp_h(t),\br_h(t))$ of problem \eqref{varFormR1-varFormR2-h} started up with the initial conditions 
\begin{equation}\label{initial-R1-R2-h*}
\begin{array}{c}
\kp_h(0)= \varXi_h\kp_0 \,,
\quad  \dot\kp_h(0) = \varXi_h\kp_1,\quad 
 \br_h(0) = Q_h\br_{0},\quad \dot\br_h(0) =  Q_h\br_{1}.
\end{array}
\end{equation}
In this way, the projected errors $\be_{\kp,h}(t) := \varXi_h\kp(t) - \kp_h(t)$ and $\be_{r,h}(t) := Q_h\br(t) - \br_h(t)$  
satisfy by construction vanishing initial conditions: 
\begin{equation}\label{inti0}
\be_{\kp, h}(0)=(\mathbf 0,\mathbf 0)\quad  \be_{r,h}(0) = \mathbf{0}, \quad 
\dot{\be}_{\kp, h}(0)= (\mathbf 0,\mathbf 0), \quad\hbox{and}\quad  \dot{\be}_{r,h}(0)= \mathbf{0}\,.
\end{equation}
Moreover, by definition of $\varXi_h$ and due to the second equations of \eqref{varFormR1-varFormR2} 
and \eqref{varFormR1-varFormR2-h}, it turns out that 
\begin{equation}\label{skew0}
	\big(\bs, \be_{\kp, h}^+(t)\big) = \big(\bs, \kp^+(t)\big) - \big(\bs, \kp^+_h(t)\big)= 0,\quad \forall \bs \in 
	\mathbb{Q}_h.
\end{equation}  
\begin{theorem}\label{errorE} 
Assume that the solution of problem \eqref{varFormR1-varFormR2}-\eqref{initial-R1-R2} satisfies  the regularity assumptions $\kp\in \cC^2(\mathfrak{S})$ and $\br\in \cC^2(\mathbb Q)$. Then, the following error estimate holds 
\begin{equation}
\begin{array}{c}
\disp
\max_{t\in [0, T]}\norm{(\kp-\kp_h)(t)}_{\mathfrak{S}} 
\,+\, \max_{t\in [0, T]}\norm{(\dot\kp-\dot\kp_h)(t)}_{0,\Omega} + \max_{t\in [0, T]}\norm{ (\dot\br - \dot \br_h)(t)}_{0,\Omega} 
\\[2ex]
\disp
 \lesssim \, \norm{\kp - \varXi_h\kp}_{\W^{2,\infty}(\mathfrak{S})}+
\norm{\br - Q_h\br}_{\W^{2,\infty}(\mathbb{L}^2(\Omega))}.
\end{array}
\end{equation}
\end{theorem}
\begin{proof}
We first observe that,  because of the regularity assumptions on $\kp$ and $\br$, we have
\begin{equation}\label{idempoi}
\begin{array}{l}
\frac{\text{d}^i}{\text{d}t^i}\varXi_h\kp(t) = \varXi_h\frac{\text{d}^i}{\text{d}t^i}\kp(t) 
\quad\text{and} \quad  \frac{\text{d}^i}{\text{d}t^i} Q_h\br(t) = Q_h \frac{\text{d}^i}{\text{d}t^i}\br(t),  \quad \forall\, i \in \{1,2\}\,, \quad \forall\, t\in [0,T].
\end{array}
\end{equation} 
Then, using that the scheme \eqref{varFormR1-varFormR2-h} is consistent with  \eqref{varFormR1-varFormR2},
and keeping in mind \eqref{div0} and the fact that $\bdiv \mathbb W_h \subset \mathbf U_h$, we readily find that 
\begin{align}\label{errorIdentity}
\begin{split}
A\big( \ddot{\be}_{\kp, h} + \tfrac{1}{\omega} \pi_1\dot{\be}_{\kp, h}, \kq\big)  + (\ddot{\be}_{r,h}, \kq^+)   
+ \big(\bdiv \be_{\kp,h}^+(t),  \bdiv\kq^+ \big)_\rho&= F(\kq), \quad \forall \kq \in \mathfrak S_h,
\end{split}
\end{align}
with
\[
F\big(\kq\big):=  A\big( \varXi_h \ddot\kp - \ddot\kp +\tfrac{1}{\omega} \pi_1(\varXi_h \dot\kp - \dot\kp), \kq\big)  +(Q_h\ddot{\br} - \ddot{\br}, \kq^+).
\]
Next, choosing  $\kq = \dot{\be}_{\kp, h}(t)$ in  \eqref{errorIdentity} and taking into account \eqref{skew0}, we deduce that 
\begin{equation}
	\dot{\mathcal{E}}\big({\be}_{\kp, h}\big)(t) + 
	 A\big( \tfrac{1}{\omega}\pi_1\dot{\be}_{\kp, h}, \pi_1\dot{\be}_{\kp, h}\big) = F(\dot{\be}_{\kp, h}).
\end{equation}
Hence,  the Cauchy-Schwarz inequality combined with \eqref{eq-extra-1}  give   
\[
\dfrac{\dot{\mathcal{E}}\big({\be}_{\kp, h}\big)}{2\sqrt{{\mathcal{E}}\big({\be}_{\kp, h}\big)}}
 \lesssim   \sum_{i=1,2}\norm{ \frac{\text{d}^i\kp}{\text{d}t^i}  - \varXi_h\frac{\text{d}^i\kp}{\text{d}t^i}}_{0,\Omega} + 
 \norm{\ddot{\br} - Q_h\ddot{\br}}_{0,\Omega},
\]
from which, integrating with respect to time and using \eqref{XihStab}, we arrive at 
\begin{equation}\label{8.1}
\max_{t\in [0, T]}{\mathcal{E}}\big({\be}_{\kp, h}\big)^{1/2}(t) \lesssim  \norm{\kp-\varXi_h\kp}_{\W^{2,\infty}(\mathfrak{L}^2(\Omega))} + \norm{\br - Q_h\br}_{\W^{2,\infty}(\mathbb{L}^2(\Omega))}.
\end{equation}
Now, it follows from  \eqref{eq-extra-1} that 
\begin{equation}\label{har0}
\max_{t\in [0, T]}\norm{ \bdiv {\be}_{\kp, h}^+(t) }_{0,\Omega} 
+ \max_{t\in [0, T]}\norm{\dot \be_{\kp, h}(t) }_{0,\Omega} \lesssim \norm{\kp-\varXi_h\kp}_{\W^{2,\infty}(\mathfrak{L}^2(\Omega))} + \norm{\br - Q_h\br}_{\W^{2,\infty}(\mathbb{L}^2(\Omega))}.
\end{equation}
Actually, as  $\be_{\kp, h}(t) = \int_0^t \dot{\be}_{\kp, h}(s)\text{d}s$, we  also have
\begin{equation}\label{har1}
\max_{t\in [0, T]}\norm{ {\be}_{\kp, h}(t) }_{\mathfrak S} 
+ \max_{t\in [0, T]}\norm{\dot \be_{\kp, h}(t) }_{0,\Omega} \lesssim \norm{\kp-\varXi_h\kp}_{\W^{2,\infty}(\mathfrak{L}^2(\Omega))} + \norm{\br - Q_h\br}_{\W^{2,\infty}(\mathbb{L}^2(\Omega))}.
\end{equation}

In order to estimate the error in the rotation $\br$ we first notice that, 
integrating once with respect to time in  \eqref{errorIdentity}, we obtain
\begin{align}\label{errorIdentityr0}
\begin{split}
(\dot{\be}_{\br,h}, \kq^+) &= -
A\big( \dot{\be}_{\kp,h}  + \tfrac{1}{\omega}\pi_1 \be_{\kp,h}, \kq\big)   
-  \int_0^t \big(\bdiv\be_{\kp,h}^+(s),  \bdiv\kq^+ \big)_\rho\, \text{d}s,
\\[1ex]
&\qquad -A\big( \dot\kp - \varXi_h \dot\kp + \tfrac{1}{\omega}\pi_1 (\kp - \varXi_h\kp), \kq\big) + A\big( \kp_1 - \varXi_h \kp_1 + \tfrac{1}{\omega}\pi_1 (\kp_0 - \varXi_h\kp_{0}), \kq\big) 
\\[1ex]
&\qquad - (\dot\br - Q_h\dot\br, \kq^+) + (\br_1 - Q_h\br_{1}, \kq^+),\quad \forall \kq \in \mathfrak S_h.
\end{split}
\end{align}
Therefore, the inf-sup condition \eqref{discInfSup}, identity \eqref{errorIdentityr0},  the Cauchy-Schwarz inequality, \eqref{contA} and \eqref{har1} provide  
\begin{equation}\label{hart}
	\beta^*\norm{\dot{\be}_{\br,h} }_{0,\Omega} \leq \sup_{\kq\in \mathfrak S_h}
\dfrac{(\dot{\be}_{\br,h}, \kq^+ )}{\norm{\kq}_{\mathfrak S}}
\lesssim 
\norm{\kp-\varXi_h\kp}_{\W^{2,\infty}(\mathfrak{L}^2(\Omega))} + \norm{\br - Q_h\br}_{\W^{2,\infty}(\mathbb{L}^2(\Omega))}.
\end{equation}
Using  the triangle inequality relatively to the splittings  $\kp - \kp_h = (\kp - \varXi_h\kp) + \be_{\kp, h}$ and $\br - \br_h = (\br - Q_h\br) + \be_{\br, h}$ of each component the error and taking into account \eqref{hart} and \eqref{har1}, we conclude that 
\begin{equation}\label{eqBb}
\begin{array}{c}
\disp
\max_{[0, T]}\norm{\kp - \kp_h}_{\mathfrak{S}} 
+ \max_{[0, T]}\norm{\dot\kp-\dot\kp_h}_{0,\Omega} + \max_{t\in [0, T]}\norm{\dot\br-\dot\br_h}_{0,\Omega} 
\\[2ex]
\disp
\lesssim \,\norm{\kp - \varXi_h\kp}_{\W^{2,\infty}(\mathfrak{S})}
 + \norm{\br - Q_h\br}_{\W^{2,\infty}(\mathbb{L}^2(\Omega))},
\end{array}
\end{equation}
which gives the result.
\end{proof}

\begin{corollary}\label{coro1}
Assume that the solution of problem \eqref{varFormR1-varFormR2}-\eqref{initial-R1-R2} is such that  $\kp \in \cC^2(\mathbb{H}^{k}(\Omega)^{2})$, $\bdiv \kp^+ \in \cC^2(\bH^{k}(\Omega))$ and $\br\in \cC^2(\mathbb{H}^{k}(\Omega))$. Then there holds
\begin{equation}\label{asympSD}
\max_{[0, T]}\norm{\dot\kp(t)-\dot\kp_h(t)}_{0,\Omega} + \max_{[0, T]}\norm{\kp(t)-\kp_h(t)}_{\mathfrak S} +\norm{\br-\br_h}_{\W^{1,\infty}(\mathbb{L}^2(\Omega))}   \lesssim  h^k.
\end{equation}
\end{corollary}
\begin{proof}
It is a direct consequence of \eqref{asymp}, \eqref{asympDiv}, \eqref{XihStab}, and Theorem \ref{errorE}.
\end{proof}

We stress here that the quantity $\ddot{\bu}_h(t) := \tfrac{1}{\rho}\big(\bdiv \kp_h^+(t) + U_h \bF(t)\big)$ 
provides a direct and accurate approximation of the acceleration field $\ddot \bu$. Indeed, under the 
assumptions of Corollary~\ref{coro1},  the triangle inequality yields
\[
\max_{t\in [0, T]}\| (\ddot\bu - \ddot \bu_h)(t)\|_{0,\Omega}\lesssim \max_{t\in [0, T]}\|\bdiv (\kp - \kp_h)^+\|_{0,\Omega} + \max_{t\in [0, T]}\|\bF - U_h \bF\|_{0,\Omega} \lesssim h^k.
\]

\section{The fully discrete scheme and its convergence analysis}\label{section7}

Given $L\in \mathbb{N}$, we consider a uniform partition of the time interval $[0, T]$ with step size $\Delta t := T/L$. Then, for any continuous function $\phi:[0, T]\to \R$ and for each $k\in\{0,1,\ldots,L\}$, we denote $\phi^k := \phi(t_k)$, where $t_k := k\,\Delta t$. In addition, we adopt the same notation for vector/tensor valued functions and consider $t_{k+\frac{1}{2}}:= \frac{t_{k+1} + t_k}{2}$, $\phi^{k+\frac{1}{2}}:= \frac{\phi^{k+1} + \phi^k}{2}$, $\phi^{k-\frac{1}{2}}:= \frac{\phi^{k} + \phi^{k-1}}{2}$, and the discrete time derivatives
\[
\partial_t \phi^k := \frac{\phi^{k+1} - \phi^k}{\Delta t}, \quad \bar \partial_t \phi^k :=  \frac{\phi^{k} - \phi^{k-1}}{\Delta t}\quad\text{and} \quad \partial^0_t \phi^k := \frac{\phi^{k+1} - \phi^{k-1}}{2\Delta t}\,,
\]
from which we notice that 
\begin{align*}
\partial_t \bar \partial_t \phi^k  = \frac{\bar \partial_t \phi^{k+1} - \bar \partial_t \phi^k}{\Delta t} 
= \frac{\partial_t \phi^{k} - \partial_t \phi^{k-1}}{\Delta t}\,.
\end{align*}

In what follows we utilize the Newmark trapezoidal rule for the time discretization of \eqref{varFormR1-varFormR2-h}-\eqref{initial-R1-R2-h} which means that, for each $k=1,\ldots,L-1$, we look for $\kp_h^{k+1}\in \mathfrak S_h$ and $\br^{k+1}_h\in \mathbb{Q}_h$ such that
\begin{align}\label{fullyDiscretePb1-Pb2}
\begin{split}
 A\Big( \partial_t\bar \partial_t\kp^k + \tfrac{1}{\omega}\pi_1 \partial^0_t \kp^k_h , \kq\Big)  
+ (\partial_t\bar \partial_t \br^k_h,\kq^+) + \Big(\bF(t_k)+ \bdiv \big(\frac{\kp_h^{k+\frac{1}{2}} + \kp_h^{k-\frac{1}{2}}}{2}\big)^+,  \bdiv \kq^+  \Big)_\rho
&=  \mathbf 0 
\\[1ex]
\big(\bs,(\kp_h^{k+1} )^+\big) &= 0
\end{split}
\end{align}
for all $\kq\in \mathfrak{S}_h$ and $\bs\in \mathbb{Q}_h$.  In addition, for the sake of simplicity, we assume that the scheme \eqref{fullyDiscretePb1-Pb2} is initiated as in   \eqref{initial-R1-R2-h*}.
It is straightforward to realize that the functions $\be^k_{\kp, h}  := \varXi_h\kp(t_k) - \kp_h^k$ 
and $\be^k_{\br,h} := Q_h\br(t_k) - \br_h^k$ solve the equations
\begin{align}\label{projectedErrorEqk}
\begin{split}
 A\Big( \partial_t\bar \partial_t \be^k_{\kp, h} + \tfrac{1}{\omega}\partial^0_t \pi_1\be^k_{\kp, h} , \kq\Big)  
+ \big(\partial_t\bar \partial_t \be^k_{\br,h}, \kq^+\big) 
+ \Big(\bdiv \big( \dfrac{\be^{k+\frac{1}{2}}_{\kp, h}  +\be^{k-\frac{1}{2}}_{\kp, h} }{2}\big)^+,  \bdiv \kq^+ \Big)_\rho
&=  G^k(\kq)
\\[1ex]
(\bs,(\be^{k+1}_{\kp, h})^+ ) &= 0,
\end{split}
\end{align}
for all $\kq\in \mathfrak{S}_h$ and $\bs\in \mathbb{Q}_h$, where 
\[
G^k(\kq) :=A\big(\mathfrak{X}_1^k, \kq \big)  +(\bchi^k, \kq^+) + \big(\bdiv (\mathfrak{X}_2^k)^+,  \bdiv \kq^+\big)_\rho,
\]
with 
\[
\mathfrak{X}_1^k := \varXi_h\partial_t\bar \partial_t \kp(t_k) - \ddot\kp(t_k) +  \pi_1\big(\varXi_h\partial^0_t\kp(t_k) - \dot\kp(t_k)\big),
\qquad
\bchi^k = Q_h\partial_t\bar \partial_t  \br(t_k) - \ddot\br(t_k),
\]
and (thanks to \eqref{div0} and because $\bdiv \mathbb W_h \subset \mathbf U_h$) 
\[
\mathfrak{X}_2^k:= \dfrac{\kp(t_{k+1}) -2\kp(t_k) + \kp(t_{k-1})}{4}.
\]

Discrete techniques mimicking those used for the semi-discrete problem in Theorem~\ref{AprioriBound_h} permit to estimate the projected errors in terms of the consistency errors as follows.
\begin{lemma}\label{fullStability}
Assume that the solution of problem \eqref{varFormR1-varFormR2}-\eqref{initial-R1-R2} satisfies $\kp \in \mathcal C^2(\mathfrak L^2(\Omega))$ and $\br \in \mathcal C^2(\mathbb Q)$. Then, the following estimate holds true
\begin{align}\label{full4}
\begin{split}
\disp
\max_n  &\norm{\partial_t \be^n_{\kp, h}}_{0,\Omega} + \max_n\norm{\bdiv (\be^n_{\kp, h})^+ }_{0,\Omega} 
+ \max_n \norm{\partial_t \be^n_{\br,h}}_{0,\Omega} \\[1ex] 
& \lesssim  \max_n\norm{\mathfrak{X}_1^n}_{0,\Omega}  + \max_n\norm{\bchi^n}_{0,\Omega}
+ \max_n\norm{\bdiv (\partial_t  \mathfrak{X}_2^n)^+}_{0,\Omega} + \max_n\norm{\bdiv (\mathfrak{X}_2^n)^+}_{0,\Omega}.
\end{split}
\end{align}
\end{lemma}
\begin{proof}
Taking  $\kq = \partial^0_t\be^{k}_{\kp, h}= \dfrac{ \be_{\kp, h}^{k+\frac{1}{2}} - \be_{\kp, h}^{k-\frac{1}{2}} }{\Delta t} = \dfrac{\partial_t \be_{\kp, h}^k + \partial_t \be_{\kp, h}^{k-1}}{2}$ in \eqref{projectedErrorEqk} yields the identity 
\begin{align*}
	\frac{1}{2 \Delta t} A\Big( \partial_t \be^k_{\kp,h} &- \partial_t \be^{k-1}_{\kp,h}, 
 \partial_t \be^k_{\kp,h} + \partial_t \be^{k-1}_{\kp,h} \Big) + 
  A\Big( \tfrac{1}{\omega}\pi_1\partial^0_t \be^k_{\kp, h}, \pi_1\partial^0_t \be^k_{\kp, h} \Big) 
 \\[1ex]
\disp
&+ \frac{1}{2 \Delta t} \Big(\bdiv \big(\be^{k+\frac{1}{2}}_{\kp,h} + 
\be^{k-\frac{1}{2}}_{\kp,h}\big)^+ , \bdiv \big(\be^{k+\frac{1}{2}}_{\kp,h} - 
\be^{k-\frac{1}{2}}_{\kp,h}\big)^+ \Big)_\rho = G^k(\partial^0_t \be^k_{\kp, h}),
\end{align*}
from which we obtain the estimate
\begin{equation*}
\begin{array}{c}
\disp
A\Big( \partial_t \be^k_{\kp,h}, 
 \partial_t \be^k_{\kp,h} \Big) -  A\Big( \partial_t \be^{k-1}_{\kp,h}, 
 \partial_t \be^{k-1}_{\kp,h} \Big)
   + \norm{\tfrac{1}{\sqrt{\rho}} \bdiv \big(\be^{k+\frac{1}{2}}_{\kp,h}\big)^+}_{0,\Omega}^2 \\[2ex]
   \disp
   - \norm{\tfrac{1}{\sqrt \rho} \bdiv \big(\be^{k-\frac{1}{2}}_{\kp,h}\big)^+}_{0,\Omega}^2  
   \leq  2 \Delta t G^k(\partial^0_t \be^k_{\kp, h}).
   \end{array}
\end{equation*}
Then, summing the foregoing inequality over $k=1,\ldots, n$ and using \eqref{ellipA} gives 
\begin{equation}\label{full2}
\norm{\partial_t \be^n_{\kp,h}}_{0,\Omega}^2 + \norm{ \bdiv \big(\be^{n+\frac{1}{2}}_{\kp,h}\big)^+}_{0,\Omega}^2 \lesssim \Delta t \sum_{k= 1}^n G^k(\partial^0_t \be^k_{\kp, h}).
\end{equation}
Performing a discrete integration by parts in the summation corresponding to the term containing $\mathfrak{X}_2$  on the right hand side of \eqref{full2}, we arrive at  
\begin{align*}
	\Delta t\sum_{k= 1}^n &G^k(\partial^0_t \be^k_{\kp, h}) = \Delta t\sum_{k= 1}^n A\Big(\mathfrak{X}_1^k, \frac{\partial_t \be_{\kp, h}^k + \partial_t \be_{\kp, h}^{k-1}}{2} \Big)  + \Delta t\sum_{k= 1}^n \Big(\bchi^k, \big(\frac{\partial_t \be_{\kp, h}^k + \partial_t \be_{\kp, h}^{k-1}}{2}\big)^+\Big)
	\\
	 &-  \Delta t\sum_{k= 1}^{n-1}\Big(\bdiv \partial_t(\mathfrak{X}_2^k)^+,  \bdiv \big( \be_{\kp, h}^{k+\frac{1}{2}} \big)^+\Big)_\rho +  \Big( \bdiv \big(\mathfrak{X}_2^n\big)^+, \bdiv \big( \be_{\kp, h}^{n+\frac{1}{2}} \big)^+\Big)_\rho.
\end{align*} 
Using this expression of the right hand side of \eqref{full2} and the Cauchy-Schwarz inequality we obtain by means of straightforward calculations the estimate 
\begin{equation}\label{full30}
\begin{array}{c}
\disp
\max_n\norm{\partial_t \be^n_{\kp,h}}_{0,\Omega} + \max_n\norm{\bdiv \be^{n+\frac{1}{2}}_{\kp,h}}_{0,\Omega} \lesssim  
\Delta t \sum_{k= 1}^L 
\norm{\mathfrak{X}_1^k}_{0,\Omega} + \Delta t \sum_{k= 1}^L \norm{\bchi^k}_{0,\Omega} 
\\[1ex]
\disp
+ \, \Delta t \sum_{k= 1}^{L} \norm{\bdiv (\partial_t  \mathfrak{X}_2^k)^+}_{0,\Omega} + 
\max_n\norm{\bdiv ( \mathfrak{X}_2^n)^+}_{0,\Omega}.
\end{array}
\end{equation}
It remains to obtain bounds in the $\L^2$-norm for the projected error in the variable $\br$. In fact, multiplying  the first equation of \eqref{projectedErrorEqk} by $\Delta t$ and summing over $k=1,\ldots,n$, we get
\begin{align*}
\disp
(\partial_t \be^{n+1}_{\br,h}, \btau) &=  -A\Big( \partial_t \be^{n+1}_{\kp, h}, \kq\Big) 
-   A\Big( \tfrac{1}{\omega} \frac{\be^{n+1}_{\kp, h}+\be^{n}_{\kp, h}}{2} , \pi_1\kq\Big) 
\\[1ex]
 &- \Delta t\sum_{k= 1}^n \Big(\bdiv \big( \dfrac{\be^{k+\frac{1}{2}}_{\kp, h}  +
\be^{k-\frac{1}{2}}_{\kp, h} }{2}\big)^+,  \bdiv \kq^+ \Big)_\rho + \Delta t\sum_{k= 1}^nG^k(\kq).
\end{align*}
Hence, by virtue of the inf-sup condition \eqref{discInfSup}, the Cauchy-Schwarz inequality, and \eqref{contA}, we have that 
\begin{align*}
	\disp 
\norm{\partial_t \be^n_{\br,h}}_{0,\Omega} \lesssim \sup_{\kq\in \mathfrak S_h} \frac{ (\partial_t \be^n_{\br,h}, \kq^+) }{\norm{\kq}_{\mathfrak S}} \lesssim  \max_n\norm{\partial_t \be^n_{\kp,h}}_{0,\Omega} + \max_n\norm{\be^n_{\kp,h}}_{0,\Omega} +\max_n\norm{\bdiv \big(\be_{\kp,h}^{n+\frac{1}{2}}\big)^+}_{0,\Omega} 
\\[1ex]
\disp
+ 	\, \max_n\norm{\mathfrak{X}_1^n}_{0,\Omega} + \max_n\norm{\bchi^n}_{0,\Omega}
+  \max_n\norm{\bdiv (\partial_t  \mathfrak{X}_2^n)^+}_{0,\Omega} + 
\max_n\norm{\bdiv (\mathfrak{X}_2^n)^+}_{0,\Omega},
\end{align*}
and the result follows from \eqref{full30} and the fact that 
\[
\norm{\be^n_{\kp,h}}_{0,\Omega} \leq \Delta t \sum_{k=1}^n \norm{\bar \partial_t \be^k_{\kp,h}}_{0,\Omega}\leq \max_k\norm{\partial_t \be^k_{\kp,h}}_{0,\Omega}\quad \forall n=1\ldots, L.
\]
\end{proof}

The stability estimate obtained in \eqref{full4} for the projected errors and  Taylor expansions 
for the different consistency terms, provide  the following quasi-optimal convergence result. 
\begin{theorem}\label{mainFD}
Assume that the solution of problem \eqref{varFormR1-varFormR2}-\eqref{initial-R1-R2} is such that 
$\kp\in \cC^4(\mathfrak S)$ and $\br\in \cC^4(\mathbb{L}^2(\Omega))$. Then, it holds that 
\begin{align}\label{fullConv1}
\begin{split}
&\max_n\norm{\kp(t_{n+\frac{1}{2}}) - \kp^n_{h}}_{0,\Omega}+
 \max_n
\norm{ \bdiv\big(\kp(t_{n+\frac{1}{2}}) - \kp^{n+\frac{1}{2}}_h\big)^+}_{0,\Omega} + 
\max_n \norm{\br(t_{n+\frac{1}{2}}) - \br^n_{h}}_{0,\Omega}
\\[1ex]
&\lesssim  
\norm{\kp - \varXi_h\kp}_{\W^{2,\infty}(\mathfrak S)} + 
\norm{\br - Q_h\br}_{\W^{2,\infty}(\mathbb{L}^2(\Omega))} + 
\Delta t^2 \Big( \norm{\bsig}_{\W^{4,\infty}(\mathfrak S)} + \norm{\br}_{\W^{4,\infty}(\mathbb L^2(\Omega))} \Big).
\end{split}
\end{align}
\end{theorem}
\begin{proof}
It follows from the triangle inequality and  the stability estimate \eqref{full4} that 
\begin{align}\label{uno}
\begin{split}
\max_n &\norm{\dot{\kp}(t_{n+\frac{1}{2}}) - \partial_t \kp^n_{h}}_{0,\Omega} + 
\max_n\norm{\bdiv \big( \kp(t_{n+\frac{1}{2}}) - \kp^{n+\frac{1}{2}}_h \big)^+}_{0,\Omega} + 
\max_n \norm{\dot{\br}(t_{n+\frac{1}{2}}) - \partial_t \br^n_{h}}_{0,\Omega}
\\[1ex]
&\leq \max_n \norm{\dot{\kp}(t_{n+\frac{1}{2}}) - \varXi_h\partial_t \kp(t_n) }_{0,\Omega} + 
\max_n\Big\|\bdiv\Big ( \kp(t_{n+\frac{1}{2}}) -  \varXi_h\frac{\kp(t_{n+1}) + \kp(t_{n}) }{2}\Big)^+ \Big\|_{0,\Omega}
\\[1ex]
&\quad +\max_n \norm{\dot{\br}(t_{n+\frac{1}{2}}) - Q_h\partial _t \br(t_n) }_{0,\Omega}  +\max_n  \norm{\partial_t \be^n_{\kp,h}}_{0,\Omega} + \max_n\norm{\bdiv \be^{n+\frac{1}{2}}_{\kp,h}}_{\rho} 
\\[1ex]
& \quad + \max_n  \norm{\partial_t \be^n_{\br,h}}_{0,\Omega} \lesssim 
\max_n \norm{\dot{\kp}(t_{n+\frac{1}{2}}) - \varXi_h\partial_t \kp(t_n) }_{0,\Omega} 
\\[1ex]
&\quad + \max_n \norm{\dot{\br}(t_{n+\frac{1}{2}}) - Q_h\partial _t \br(t_n) }_{0,\Omega} + 
\max_n\Big\|\bdiv\Big ( \kp(t_{n+\frac{1}{2}}) -  \varXi_h\frac{\kp(t_{n+1}) + \kp(t_{n}) }{2}\Big)^+ \Big\|_{0,\Omega}
\\[1ex]
&\quad +  
\max_n\norm{\mathfrak{X}_1^n}_{0,\Omega} + \max_n\norm{\bchi^n}_{0,\Omega} + \max_n\norm{\bdiv (\partial_t  \mathfrak{X}_2^n)^+}_{0,\Omega} + \max_n\norm{\bdiv (\mathfrak{X}_2^n)^+}_{0,\Omega}.
\end{split}
\end{align}
Centering the following Taylor expansions at $t=t_n$ gives  
\begin{align}\label{chi1}
\begin{split}
\mathfrak{X}_1^n = \varXi_h\ddot{\kp}(t_n) &- \ddot{\kp}(t_n) + 
\frac{\Delta t^2}{6 } \int_{-1}^{1} (1-|s|)^3\varXi_h\dfrac{\text{d}^4 \kp}{\text{d}t^4}(t_n + \Delta t\, s)  \, \text{d}s  
\\
 &+ \pi_1\Big( \varXi_h\dot{\kp}(t_n) - \dot{\kp}(t_n) \Big)+ \frac{\Delta t^2}{2 } \int_{-1}^{1} (1 - |s|)^2 \pi_1\varXi_h\dfrac{\text{d}^3\, \kp}{\text{d}t^3}(t_n + \Delta t\, s)\, \text{d}s \,,
 \end{split}
\end{align}
\begin{equation}\label{nuevo}
\bchi^n = Q_h\ddot{\br}(t_n) - \ddot{\br}(t_n) + 
\frac{\Delta t^2}{6 } \int_{-1}^{1} (1-|s|)^3Q_h\dfrac{\text{d}^4 \br}{\text{d}t^4}(t_n + \Delta t\, s)
 \, \text{d}s \,,
\end{equation}
\begin{equation}\label{chi2}
\mathfrak{X}_2^n = \frac{\Delta t^2}{4} \int_{-1}^{1} (1 - |s|) \ddot{\kp}(t_n + \Delta t\, s)  \, \text{d}s\,,
\end{equation}
and
\begin{align}\label{diffchi2}
\begin{split}
\partial_t  \mathfrak{X}_2^k &= \frac{\kp(t_{n+2}) - 3 \kp(t_{n+1}) + 3\kp(t_{n}) -\kp(t_{n-1})}{4\Delta t}
\\
&= \Delta t^2 \int_0^1 (1-s)^2  \Big( \dfrac{\text{d}^3 \kp}{\text{d}t^3} (t_n + 2\Delta t\, s) -\dfrac{3}{8}\dfrac{\text{d}^3 \kp}{\text{d}t^3} (t_n + \Delta t\, s) + \dfrac{1}{8}\dfrac{\text{d}^3 \kp}{\text{d}t^3} (t_n - \Delta t\, s)\Big) \, \text{d}s.
\end{split}
\end{align}
Expanding this time about $t=t_{n+\frac{1}{2}}$ gives
\begin{equation}\label{Taylor1}
\kp(t_{n+\frac{1}{2}}) -  \varXi_h\frac{\kp(t_{n+1}) + \kp(t_{n}) }{2} = \kp(t_{n + \frac{1}{2}}) - \varXi_h\kp(t_{n + \frac{1}{2}}) - 
\frac{\Delta t^2}{4} 
\int_{-1}^{1} (1-|s|)\varXi_h\ddot{\kp}(t_{n+\frac{1}{2}} + \frac{\Delta t}{2}s)\, \text{d}s,
\end{equation}
\begin{equation}\label{Taylor2}
\dot{\kp}(t_{n+\frac{1}{2}}) - \varXi_h\partial _t \kp(t_n) = \dot{\kp}(t_{n + \frac{1}{2}}) - \varXi_h\dot{\kp}(t_{n + \frac{1}{2}}) - \frac{\Delta t^2}{8} \int_{-1}^{1} (1 - |s|)^2 \varXi_h\dfrac{\text{d}^3 \kp}{\text{d}t^3} (t_{n+\frac{1}{2}} + \frac{\Delta t}{2}s) \, \text{d}s,
\end{equation}
and
\begin{equation}\label{Taylor2r}
\dot{\br}(t_{n+\frac{1}{2}}) - \partial _t \br^*_h(t_n) = \dot{\br}(t_{n + \frac{1}{2}}) - Q_h\dot{\br}(t_{n + \frac{1}{2}}) - \frac{\Delta t^2}{8} \int_{-1}^{1} (1 - |s|)^2Q_h\dfrac{\text{d}^3 \br}{\text{d}t^3} (t_{n+\frac{1}{2}} + \frac{\Delta t}{2}s) \, \text{d}s.
\end{equation}
Using that $\varXi_h: \mathfrak S \to \mathfrak S_h$ and $Q_h: \mathbb Q \to \mathbb Q_h$ are uniformly bounded in $h$ and taking advantage of  \eqref{chi1}, \eqref{chi2} and \eqref{diffchi2},  we readily obtain the bound   
\begin{align}\label{consistency1}
\begin{split}
&\max_n\norm{\mathfrak{X}_1^n}_{0,\Omega} + \max_n\norm{\bchi^n}_{0,\Omega}
+  
\max_n\norm{\bdiv (\partial_t  \mathfrak{X}_2^n)^+}_{0,\Omega} + 
\max_n\norm{\bdiv (\mathfrak{X}_2^n)^+}_{0,\Omega}  
\\[1ex]
&\quad \lesssim  \norm{\br - Q_h \br}_{\W^{2,\infty}(\mathbb{L}^2(\Omega))}+\norm{\kp - \varXi_h\kp}_{\W^{2,\infty}(\mathfrak S)}+ 
 \Delta t^2 \Big( 
\norm{\br}_{\W^{4,\infty}(\mathbb{L}^2(\Omega))} + \norm{\kp}_{\W^{4,\infty}(\mathfrak S)}
\Big).
\end{split}
\end{align} 
On the other hand, \eqref{Taylor1}, \eqref{Taylor2} and \eqref{Taylor2r} yield 
\begin{align}\label{consistency2}
\begin{split}
\max_n &\norm{\dot{\kp}(t_{n+\frac{1}{2}}) - \varXi_h \partial _t\kp(t_n) }_{0,\Omega} + 
\max_n \norm{\dot{\br}(t_{n+\frac{1}{2}}) - Q_h\partial _t \br(t_n) }_{0,\Omega}
\\[1ex] 
&+\max_n\big\|\bdiv\Big ( \kp(t_{n+\frac{1}{2}}) -  \varXi_h\frac{\kp(t_{n+1}) + \kp(t_{n}) }{2}\Big)^+\big\|_{0,\Omega} 
\lesssim 
\norm{\kp - \varXi_h\kp}_{\W^{1,\infty}(\mathfrak S)}
\\ 
& +  \norm{\br - Q_h\br}_{\W^{1,\infty}(\H(\bdiv, \Omega))} +\Delta t^2 \Big( 
\norm{\br}_{\W^{3,\infty}(\mathbb{L}^2(\Omega))} + \norm{\bsig}_{\W^{3,\infty}(\mathfrak S)}
\Big).
\end{split}
\end{align} 
Combining \eqref{consistency1}, \eqref{consistency2} with \eqref{uno} we obtain
\begin{align}\label{dos}
\begin{split}
\max_n\norm{\dot{\kp}(t_{n+\frac{1}{2}}) - \partial_t \kp^n_{h}}_{0,\Omega}+
 \max_n
\norm{ \bdiv\big(\kp(t_{n+\frac{1}{2}}) - \kp^{n+\frac{1}{2}}_h\big)^+}_{0,\Omega} + 
\max_n \norm{\dot{\br}(t_{n+\frac{1}{2}}) - \partial_t \br^n_{h}}_{0,\Omega} \\
\lesssim 
\norm{\kp - \varXi_h\kp}_{\W^{2,\infty}(\mathfrak S)} + 
\norm{\br - Q_h\br}_{\W^{2,\infty}(\mathbb{L}^2(\Omega))} + \Delta t^2 \Big( \norm{\kp}_{\W^{4,\infty}(\mathfrak S)} + 
\norm{\br}_{\W^{4,\infty}(\mathbb{L}^2(\Omega))} 
\Big).
\end{split}
\end{align}
Finally, to obtain error estimates for $\kp(t_{n+\frac{1}{2}}) - \kp_h^{n+\frac{1}{2}}$ we sum the identity 
\begin{align}\label{truco}
\begin{split}
(\kp(t_{k+\frac{1}{2}}) &- \kp_h^{k+\frac{1}{2}} ) - (\kp(t_{k-\frac{1}{2}}) - \kp_h^{k-\frac{1}{2}}) = \kp(t_{k+\frac{1}{2}})
- \kp(t_{k-\frac{1}{2}}) - \frac{\Delta t}{2}(\dot{\kp}(t_{k+\frac{1}{2}}) + \dot{\kp}(t_{k-\frac{1}{2}}))
\\[1ex] 
&+ \frac{\Delta t}{2} (\dot{\kp}(t_{k-\frac{1}{2}}) - \partial_t \kp_h^{k-1}  )
  + 
\frac{\Delta t}{2} (\dot{\kp}(t_{k+\frac{1}{2}}) - \partial_t \kp_h^k  )
\\[1ex]
&= \frac{\Delta t^3}{16} \int_{-1}^{1} (s^2 -1)\frac{\text{d}^3\kp}{\text{d}t^3}(t_k + \frac{\Delta t}{2} s) \, \text{d}s + \frac{\Delta t}{2} (\dot{\kp}(t_{k-\frac{1}{2}}) - \partial_t \kp_h^{k-1}  )  + 
\frac{\Delta t}{2} (\dot{\kp}(t_{k+\frac{1}{2}}) - \partial_t \kp_h^k  )
\end{split}
\end{align}
over $k= 1,\ldots, n$ to deduce that
\begin{align*}
\max_{n}&\norm{\kp(t_{n+\frac{1}{2}}) - \kp_h^{n+\frac{1}{2}}}_{0,\Omega} 
\lesssim  \Delta t^2  \norm{\kp}_{\W^{3,\infty}(\mathfrak{L}^2(\Omega))}
 +\max_{n}\norm{\dot{\kp}(t_{n+\frac{1}{2}}) - \partial_t \kp^n_{h}}_{0,\Omega}. 
\end{align*}
Similar manipulations yield 
\begin{align*}
 \max_{n}\norm{\br(t_{n+\frac{1}{2}}) - \br_h^{n+\frac{1}{2}}}_{0,\Omega}
\lesssim  
  \Delta t^2 \norm{\br}_{\W^{3,\infty}(\mathbb{L}^2(\Omega))} 
 + \max_{n}\norm{\dot{\br}(t_{n+\frac{1}{2}}) - \partial_t \br^n_{h}}_{0,\Omega}.
\end{align*}
The result is now a direct consequence of the last two estimates and \eqref{dos}.
\end{proof}

Under adequate time and space regularity assumptions, we obtain the following asymptotic error estimate.
\begin{corollary}\label{coro1dt}
Assume that the solution of \eqref{varFormR1-varFormR2}-\eqref{initial-R1-R2} 
is such that  $\kp\in \cC^4(\mathfrak S)\cap \cC^2\big([\mathbb{H}^{k}(\Omega)]^2\big)$, $\bdiv \kp^+ \in \cC^2(\bH^{k}(\Omega))$ and $\br\in \cC^4(\mathbb{L}^2(\Omega))\cap \cC^2(\mathbb{H}^k(\Omega))$. Then there holds 
\begin{equation}\label{asympSDdt}
\max_n\norm{\kp(t_{n+\frac{1}{2}}) -  \kp^{n+\frac{1}{2}}_{h}}_{\mathfrak S} + 
 \max_n \norm{\br(t_{n+\frac{1}{2}}) -  \br^{n+\frac{1}{2}}_{h}}_{0,\Omega} \lesssim h^k + \Delta t^2.
\end{equation}
\end{corollary}
\begin{proof}
The result is a direct consequence of  \eqref{asymp}, \eqref{asympDiv}, \eqref{XihStab} and Theorem~\ref{mainFD}.
\end{proof}

 We remark  that,  as in the semi-discrete case, an accurate approximation of the acceleration field 
$\ddot{\bu} = \tfrac{1}{\rho}\big( \bF + \bdiv \kp^+\big)$ can be obtained by a local postprocessing  with 
\begin{equation}\label{aceleracion}
\ddot{\bu}_h^{n} \,:=\, \frac{1}{\rho}\,\bigg\{
\bdiv\Big( \dfrac{\kp_h^{n+1/2} + \kp_h^{n-1/2}}{2} \Big)^+ 
\,+\, U_h \bF(t_{n})\bigg\} \,,
\qquad n = 1, \ldots , L-1\,.
\end{equation}
It is straightforward to deduce from Corollary~\ref{coro1dt} that   
\[
\max_{1\leq n\leq L-1}\| \ddot\bu(t_n) - \ddot{\bu}_h^{n}\|_{0,\Omega}\lesssim  h^k + \Delta t^2.
\]

\section{Numerical results}\label{section8}

In this section  we show that the numerical rates of convergence delivered by the fully discrete scheme \eqref{fullyDiscretePb1-Pb2} are in accordance with the theoretical ones. For simplicity, we restrict our tests to two-dimensional model problems and assume that the medium is isotropic, namely, we assume that the tensors $\mathcal C$ and $\mathcal D$ are given by 
\[
\mathcal C \btau = 2\mu\btau + \lambda \tr(\btau)\boldsymbol{I} \quad\text{and}\quad
\mathcal D \btau = 2 a \mu\btau +  b \lambda \tr(\btau)\boldsymbol{I},
\]
with coefficients $\mu>0$, $\lambda>0$, $a>1$, and $b>1$. 

\paragraph{Convergence test.}
In the first example, we set $\Omega=(0,1)\times (0,1)$, $T= 1$ and consider an  homogeneous medium with $\rho = 1$, $\mu=\lambda=1$, $a=b=3$,  and with a constant relaxation time $\omega = 1$. We select the source $\bF$ in such a way that the exact solution is given by 
\begin{equation}\label{eq-sol-1}
\boldsymbol{u}(\boldsymbol{x},t)= \begin{pmatrix}
	(1- x_1)x_1^2 \sin(\pi x_2) \cos t 
	\\[1ex]
	(1+t) \sin(\pi x_1) \sin(\pi x_2)
\end{pmatrix}  \quad \forall \boldsymbol{x} := (x_1,x_2) \in \Omega, \quad \forall t \in [0,T].
\end{equation}
The data necessary to initiate \eqref{fullyDiscretePb1-Pb2} are deduced directly from the exact solution. 

The numerical results  presented in Table \ref{table1} correspond to a space discretization based on the second order AFW element for a sequence of nested uniform triangular meshes $\cT_h$ of $\Omega$.  For each mesh size $h$ we take $\Delta t = h$ and the individual relative errors produced by the fully discrete method  \eqref{fullyDiscretePb1-Pb2} are measured at the final time step as follows: 
\[
\texttt{e}_h(\kp) := \frac{\| \kp(t_{L-\frac{1}{2}}) - \kp^{L-\frac{1}{2}}_h \|_{\mathfrak S}}{\| \kp(t_{L-\frac{1}{2}}) \|_{\mathfrak S}},
\qquad 
\texttt{e}_h(\br) := \frac{\| \br(t_{L-\frac{1}{2}}) - \br^{L-\frac{1}{2}}_h \|_{0,\Omega}}{\| \br(t_{L-\frac{1}{2}}) \|_{0,\Omega}}\,,
\]
\[
{\texttt{e}}_h(\ddot\bu) := \frac{\| \ddot{\bu}(t_{L-1}) - \ddot{\bu}_h^{L-1} \|_{0,\Omega}}{\| \ddot{\bu}(t_{L-1}) \|_{0,\Omega}}\,,
\]
where  $(\kp,\br)$  and $(\kp^k_h,\br_h^k)\,, \, k =0, \ldots , L$, are the solutions of  \eqref{varFormR1-varFormR2} and  \eqref{fullyDiscretePb1-Pb2}, respectively. The approximation $\ddot{\bu}_h^{L-1}$ of the acceleration at $t=t_{L-1}$ is obtained from formula \eqref{aceleracion}. Additionally, we introduce the experimental rates of convergence
\[
\texttt{r}_h(\star) := \frac{\log(  \texttt{e}_h(\star)/ 
 \texttt{e}_{\widehat h}(\star) )}{\log (h/\widehat{h})} \qquad
 \forall\, \star \in \big\{ \kp, \br, \ddot\bu\big\}\,,
\]
where $\texttt{e}_h$ and ${\texttt{e}}_{\widehat h}$ are the errors corresponding to two consecutive triangulations with mesh sizes $h$ and $\widehat{h}$, respectively.  We observe there that the expected quadratic convergence rate of the error is attained in each variable. 

\begin{table}[ht]
	\centering
\small
		\begin{tabular}{l c c c c c c c c c c c c}
			\toprule
			$h = \Delta t$ & $\verb"e"_h(\kp)$ & $\verb"r"_h(\kp)$  & $\verb"e"_h(\textbf{r})$ & $\verb"r"_h(\textbf{r})$ & $\texttt{e}_h(\ddot\bu)$ & $\texttt{r}_h(\ddot\bu)$ \cr
			\toprule
			1/8   & 1.06e$-$02 & $-$ & 1.74e$-$02 & $-$ & 2.57e$+$01 & $-$ \cr
			
			1/16  & 2.44e$-$03 & 2.12 & 3.95e$-$03 & 2.14 & 6.48e$+$00 & 1.99 \cr
			
			1/32  & 5.89e$-$04 & 2.05 & 9.69e$-$04 & 2.03 & 1.63e$+$00 & 1.99 \cr
			
			1/64  & 1.46e$-$04 & 2.01 & 2.37e$-$04 & 2.03 & 4.19e$-$01 & 1.96 \cr
			
			1/128 & 3.61e$-$05 & 2.02 & 5.86e$-$05 & 2.01 & 1.06e$-$01 & 1.99 \cr
			\bottomrule
		\end{tabular}
\caption{Convergence history for the AFW element of second order with $\Delta t = h$ and coefficients $\mu= \lambda = 1$, $a=b=3$, $\rho=\omega=1$. The exact solution is given by \eqref{eq-sol-1}.}
\label{table1}
\end{table} 

\paragraph{Locking test.}
We point out that the mixed finite element method given in \cite{ggm-JSC-2017} for the elastodynamic problem (and which is extended here for viscoelasticity) was shown to be free from volumetric locking in the nearly incompressible case. Here, we carry out experiments to test the performance of the method for viscoelasticity when $\lambda >\!\! > \mu$. We maintain the same settings established in the former example and the same exact solution. We only change the values of the coefficients $\lambda$ and $\mu$ that are now chosen as $(\lambda, \mu) = (1.5\times 10^2, 3)$ in a first test and $(\lambda, \mu) = (1.5\times 10^4, 3)$ in a second one. Their corresponding Poisson's ratios are given by $\nu \simeq 0.49$ and $\nu \simeq 0.4999$, respectively.

We observe from Table~\ref{table2} that there is no degeneration  of the convergence rates  as  Poisson's ratio $\nu$ approaches the incompressible limit $0.5$. This seems to indicate  that the scheme \eqref{fullyDiscretePb1-Pb2} is inmune to locking phenomenon in the nearly incompressible case.

\begin{table}[!htb]
\centering
\small
\begin{tabular}{p{1.5cm} c c c c | c c c c} \toprule
  & \multicolumn{4}{c}{$\nu \simeq 0.49$}  &  \multicolumn{4}{c}{$\nu \simeq 0.4999$} 
 \\ \toprule
 $h = \Delta t$ & $\verb"e"_h(\kp)$ & $\verb"r"_h(\kp)$ & $\verb"e"_h(\textbf{r})$ & $\verb"r"_h(\textbf{r})$ & $\verb"e"_h(\kp)$ & $\verb"r"_h(\kp)$ & $\verb"e"_h(\textbf{r})$ & $\verb"r"_h(\textbf{r})$ 
 \\ \toprule
1/8  & 8.76e$-$03 & $-$ & 2.85e$-$02  & $-$ & 8.79e$-$03 & $-$ & 2.44e$+$00 & $-$ 
\\ 
1/16 & 1.78e$-$03 & 2.30 & 3.49e$-$03 & 3.03  & 1.78e$-$03 & 2.30 & 2.11e$-$01 & 3.53 
\\ 
1/32 & 4.33e$-$04 & 2.04 & 7.19e$-$04 & 2.28   & 4.34e$-$04 & 2.04 & 2.54e$-$02 & 3.06 
\\ 
1/64  & 1.07e$-$04 & 2.02 & 1.70e$-$04 & 2.08  & 1.07e$-$04 & 2.01 & 3.34e$-$03 & 2.92 
\\ 
1/128 & 2.65e$-$05 & 2.02 & 4.18e$-$05 & 2.03  & 2.65e$-$05 & 2.02 & 4.33e$-$04 & 2.95 
\\ \bottomrule 
\end{tabular}
\caption{Convergence history for the AFW element of second order with $h = \Delta t$ and coefficients $\mu = 3$, $a = b = 3$, $\rho = \omega = 1$. The results listed on the left and right sides of the table correspond to $\lambda = 1.5\times 10^2$ and $\lambda = 1.5\times 10^4$, respectively. The exact solution is given by \eqref{eq-sol-1}.}
\label{table2}
\end{table}

\paragraph{Piecewise constant relaxation time.} Real materials can be modelled by allowing different relaxation times in different parts of the viscoelastic body. In this experiment, we test the sensibility of our scheme regarding to jumps in the relaxation function $\omega(\boldsymbol{x})$. To this end, we let $\omega(\boldsymbol{x}) = 	\widehat{\omega}$ in $\Omega_1:=(0,1)\times (0, 1/2)$ and $\omega(\boldsymbol{x}) = 1$ in $\Omega_2:=(0,1)\times (1/2, 1)$ and fix the values $T=1$, $\rho=1$, $\mu=\lambda=1$, and $a=b=3$ for the rest of coefficients.  Non-homogeneous transmission conditions
\[
\ddot\bu|_{\Omega_2} - \ddot\bu|_{\Omega_1} = \boldsymbol{g}_1(\widehat{\omega}) \quad \text{and} \quad \kp^+|_{\Omega_2}\bn_{\Sigma}  - \kp^+|_{\Omega_1}\bn_{\Sigma} = \boldsymbol{g}_2(\widehat{\omega})  \quad \text{on $\Sigma= \partial \Omega_1\cap \partial \Omega_2$}
\]
are considered on the interface $\Sigma$, where $\bF(\widehat{\omega})$, $\boldsymbol{g}_1(\widehat{\omega})$, and $\boldsymbol{g}_2(\widehat{\omega})$ are chosen in such a way that the exact solution is still given by \eqref{eq-sol-1}. Here, $\bn_{\Sigma}$ stands for the unit normal vector on $\Sigma$ oriented towards $\Omega_2$. 

\begin{table}[!htb]
\centering
\small
\begin{tabular}{p{1cm} c c| c c | c c| c c | c c } \toprule
  & \multicolumn{2}{c}{$\widehat{\omega} = 100$}  &  \multicolumn{2}{c}{$\widehat{\omega} = 10$} 
  & \multicolumn{2}{c}{$\widehat{\omega} = 1$}  &  \multicolumn{2}{c}{$\widehat{\omega} = 0.1$} &  \multicolumn{2}{c}{$\widehat{\omega} = 0.05$}
 \\ \toprule
 ${h = \Delta t}$ & $\verb"e"_h(\kp)$ & $\verb"r"_h(\kp)$ & $\verb"e"_h(\kp)$ & $\verb"r"_h(\kp)$ & $\verb"e"_h(\kp)$ & $\verb"r"_h(\kp)$ & $\verb"e"_h(\kp)$ & $\verb"r"_h(\kp)$  & $\verb"e"_h(\kp)$ & $\verb"r"_h(\kp)$ 
 \\ \toprule
1/8  & 3.44e$-$02 & $-$ & 1.13e$-$02 & $-$ & 1.06e$-$02 & $-$ & 3.97e$-$02 & $-$ & 7.80e$-$02 & $-$
\\ 
1/16 & 7.26e$-$03 & 2.24 & 2.58e$-$03 & 2.13  & 2.44e$-$03 & 2.12 & 2.24e$-$02 & 0.83 & 6.48e$-$02 & 0.27
\\ 
1/32 & 1.03e$-$03 & 2.82 & 6.00e$-$04 & 2.11  & 5.89e$-$04 & 2.05 & 7.08e$-$03 & 1.66 & 2.44e$-$02 & 1.41
\\ 
1/64  & 2.53e$-$04 & 2.02 & 1.49e$-$04 & 2.00  & 1.46e$-$04 & 2.01 & 2.90e$-$03 & 1.29 & 1.09e$-$02 & 1.16 
\\ 
1/128 & 5.01e$-$05 & 2.34 & 3.70e$-$05 & 2.01  & 3.61e$-$05 & 2.02 & 1.09e$-$03 & 1.41 & 4.31e$-$03 & 1.35
\\ \bottomrule 
\end{tabular}
\caption{Individual relative errors for the AFW element of second order with  coefficients $\lambda = \mu = 1$, $a = b = 3$, and $\rho = 1$. The relaxation time $\omega$ is equal to $\widehat{\omega}$ in $\Omega_1$ and $1$ in $\Omega_2$. The exact solution is given by (8.1).}
\label{table3}
\end{table}

Table~\ref{table3} displays the convergence history of the variable $\kp$ for different values of $\widehat{\omega}$.  We notice that the quadratic convergence remains unaltered when we allow the material in  $\Omega_1$ to relax more quickly (by at least two orders of magnitude) with respect to the one represented by the upper half of $\Omega$. In turn, the convergence rate worsens when $\hat \omega$ is too small because our stability estimates depend on $\norm{\frac{1}{\omega}}_{\L^{\infty}(\Omega)}$. For this same reason, the method presented in this paper cannot deal with materials that are purely elastic in parts of the domain ($\omega(\boldsymbol{x})$ should  vanish  identically there). The important issue of elastic-viscoelastic composite structures will be addressed in a forthcoming work.

\end{document}